\documentclass[a4paper,12pt]{amsart}
\usepackage{amssymb}
\usepackage{cite}
\usepackage{ifthen}
\usepackage[dvips]{graphicx}
\nonstopmode \numberwithin{equation}{section}
\newtheorem*{theoA}{Theorem A}
\newtheorem*{theoB}{Theorem B}
\newtheorem*{theoC}{Theorem C}

\newtheorem*{lemA}{Lemma A}
\newtheorem*{lemB}{Lemma B}
\newtheorem*{lemC}{Lemma C}

\usepackage{amssymb}
\usepackage{ifthen}
\usepackage{graphicx}
\usepackage{amsmath}
\usepackage[T1]{fontenc} %skandit
\usepackage[utf8]{inputenc}
\usepackage[usenames,dvipsnames]{color}
\usepackage{color}
\usepackage[english]{babel}
\usepackage{fancyhdr}
\usepackage{fancybox}
\usepackage{tikz}
\theoremstyle{plain}
\newtheorem{prop}{Proposition}

\newtheorem{ques}{Question}

\newtheorem{conj}{Conjecture}

\theoremstyle{definition}

\newtheorem{exm}{Example}[section]
\newtheorem{cor}{Corollary}[section]
\newtheorem{thm}{Theorem}[section]

\newtheorem{lem}{Lemma}[section]
\newtheorem{prob}{Problem}
\newtheorem{rem}{Remark}[section]

\theoremstyle{plain}

\newcounter{minutes}
\divide\time by 60
\newcounter{hours}
\multiply\time by 60
\addtocounter{minutes}{-\time}
\newcounter {own}
\def\theown {\thesection       .\arabic{own}}

\newenvironment{pf}[1][]{%
	\vskip 3mm
	\noindent
	\ifthenelse{\equal{#1}{}}%
	{{\slshape Proof. }}%
	{{\slshape #1.} }%
}%
{\qed\bigskip}

\newcounter{alphabet}

\def\be{\begin{equation}}
	\def\ee{\end{equation}}

\newcommand{\bee}{\begin{enumerate}}
	\newcommand{\eee}{\end{enumerate}}

\newcommand{\blem}{\begin{lem}}
	\newcommand{\elem}{\end{lem}}
\newcommand{\bthm}{\begin{thm}}
	\newcommand{\ethm}{\end{thm}}
\newcommand{\bcor}{\begin{cor}}
	\newcommand{\ecor}{\end{cor}}
\newcommand{\beg}{\begin{examp}}
	\newcommand{\eeg}{\end{examp}}
\newcommand{\begs}{\begin{examples}}
	\newcommand{\eegs}{\end{examples}}

\newcommand{\bdefn}{\begin{defn}}
	\newcommand{\edefn}{\end{defn}}

\newcommand{\bprob}{\begin{prob}}
	\newcommand{\eprob}{\end{prob}}
\newcommand{\bei}{\begin{itemize}}
	\newcommand{\eei}{\end{itemize}}

\newcommand{\bcon}{\begin{conj}}
	\newcommand{\econ}{\end{conj}}
\newcommand{\bcons}{\begin{conjs}}
	\newcommand{\econs}{\end{conjs}}
\newcommand{\bprop}{\begin{prop}}
	\newcommand{\eprop}{\end{prop}}
\newcommand{\br}{\begin{rem}}
	\newcommand{\er}{\end{rem}}
\newcommand{\brs}{\begin{rems}}
	\newcommand{\ers}{\end{rems}}
\newcommand{\bo}{\begin{obser}}
	\newcommand{\eo}{\end{obser}}
\newcommand{\bos}{\begin{obsers}}
	\newcommand{\eos}{\end{obsers}}
\newcommand{\bpf}{\begin{pf}}
	\newcommand{\epf}{\end{pf}}
\newcommand{\ba}{\begin{array}}
	\newcommand{\ea}{\end{array}}
\newcommand{\beq}{\begin{eqnarray}}
	\newcommand{\beqq}{\begin{eqnarray*}}
		\newcommand{\eeq}{\end{eqnarray}}
	\newcommand{\eeqq}{\end{eqnarray*}}

\begin{document}

\title{Existence of Solutions to Systems of General Quadratic Functional Equations in $\mathbb{C}^n$}

\author{Molla Basir Ahamed}
\address{Molla Basir Ahamed,
	Department of Mathematics,
	Jadavpur University,
	Kolkata-700032, West Bengal, India.}
\email{mbahamed.math@jadavpuruniversity.in}

\author{Sanju Mandal}
\address{Sanju Mandal,
	Department of Mathematics,
	Jadavpur University,
	Kolkata-700032, West Bengal, India.}
\email{sanjum.math.rs@jadavpuruniversity.in, sanju.math.rs@gmail.com}

\subjclass[{AMS} Subject Classification:]{39A45, 30D35, 35M30, 32W50}

\def\thefootnote{}
\footnotetext{ {\tiny File:~\jobname.tex,
printed: \number\year-\number\month-\number\day,
          \thehours.\ifnum\theminutes<10{0}\fi\theminutes }
} \makeatletter\def\thefootnote{\@arabic\c@footnote}\makeatother

\begin{abstract}
The main objective of this study is to investigate the existence and forms of solutions of systems of general quadratic functional equations in $\mathbb{C}^n$. By utilizing Nevanlinna theory in $\mathbb{C}^n$, we explore the existence and form of solutions for the several systems of general quadratic difference and partial differential-difference equations of the form $af^2 + 2\alpha fg + bg^2 + 2\beta f + 2\gamma g + C=0$, where $f$ and $g$ are non-constant meromorphic functions in $\mathbb{C}^n$. The obtained results in this article are improvements and generalizations of several results from [\textit{RACSAM}, \textbf{116}(8) (2022)]. Furthermore, appropriate remarks and illustrative examples are provided to validate and demonstrate the applicability of the obtained results concerning the existence and forms of solutions for such systems of equations.
\end{abstract}

\maketitle
\pagestyle{myheadings}
\markboth{Molla Basir Ahamed and Sanju Mandal}{solutions of several systems of algebraic functional equations in $\mathbb{C}^n$}

\section{\bf Introduction}
This article is devoted to exploring the transcendental solutions for several systems of general quadratic difference or partial differential-difference equations. In recent years, extensive research has focused on finding finite order entire and meromorphic solutions in one and several complex variables for the functional equation 
\begin{align*}
	f^n(z) + g^n(z)= 1\; \mbox{for integer}\; n\geq 1
\end{align*}
regard as the Fermat diophantine equation $ x^n+y^n=1 $ over functional fields, where $ n\geq 2 $ is an integer. The study of solutions and their different properties of Circular-type functional equations $ f^2(z)+g^2(z)=1 $ in $\mathbb{C}^n$ is an active area of research that has received significant attention in recent times. Before going into a detailed information of the study of such equations in $\mathbb{C}^n$, let us revisit a crucial result concerning entire and meromorphic solutions of Circular-type equations in $\mathbb{C}^n$.
\begin{theoA}(\cite{Saleeby-AM-2013})
	For $ h:\mathbb{C}^n\rightarrow \mathbb{C} $ entire, the solutions of the Circular equation $ f^2(z) +g^2(z)=1 $ are characterize as follows:
	\begin{enumerate}
		\item [(i)] the entire solutions are $ f(z)=\cos(h(z)),\;g(z)=\sin(h(z))$;
		\item[(ii)] the meromorphic solutions are of the form $ f(z)=\frac{1-\beta^2(z)}{1+\beta^2(z)},\;g=\frac{2\beta(z)}{1+\beta^2(z)} $, with $ \beta $ being meromorphic on $\mathbb{C}^n$.
	\end{enumerate}
\end{theoA}
The study of general quadratic equations holds significance in mathematics and science because of its applications, making it a crucial area of exploration. Such equations commonly appear in problems associated with optimization, control theory, physics, and engineering, playing an essential role in the development of contemporary technology and science. Over the functional filed, it is difficult to find entire and meromorphic solutions for a algebraic $PDE$. Utilizing Nevanlinna theory in several complex variables and employing complex analysis methods, there was some literature focusing on solutions of the $PDEs$ and their many variants.

\subsection{Background of Circular-type equations and their solutions in several complex variables:}
In $1995$, Khavinson (\cite{Khavinson & Am. Math. Mon & 1995}) showed that entire solutions of the partial differential equation $\left(\frac{\partial u}{\partial z_1}\right)^2+\left(\frac{\partial u}{\partial z_2}\right)^2=1$ must be linear, i.e., $u(z_1, z_2)=az_1+bz_2+c$, where $a,b,c\in\mathbb{C}$ and $a^2+b^2=1$. Khavinson’s work inspired several researchers on finding the entire solutions to the equations like $ u_{z_1}^2+u_{z_2}^2=1 $ in $ \mathbb{C}^2 $. Later, in $2018$, Xu and Cao (see \cite{XU-CAO-MJM-2018, Xu-Cao-Correction-MJM-2018}) first obtained the following results by exploring the finite order solutions of Circular-type differential-difference equations in $ \mathbb{C}^2 $, and also to that of difference equations in $ \mathbb{C}^n $. 
\begin{theoB}(\cite[Theorem 1.2]{XU-CAO-MJM-2018})
Let $ c=(c_1, c_2) $ be a constant in $ \mathbb{C}^2 $. Then any transcendental entire solution with finite order of the partial difference-differential equation of the Fermat type
\begin{align*}
	\left(\frac{\partial f(z)}{\partial z_1}\right)^2+f(z+c)^2=1
\end{align*}
has the form of $ f(z_1, z_2)=\sin (Az_1+Bz_2 +H(z_2)) $, where $ A, B $ are constants on $ \mathbb{C} $ satisfying $A^2=1$ and $ Ae^{i(Ac_1 +Bc_2)}=1 $, $H(z_2)$ is a polynomial in one variable $z_2$ such that $H(z_2)=H(z_2+c_2)$. In the special case whenever $ c_2\neq 0 $, we have $ f(z_1, z_2)=\sin(Az_1+Bz_2+constant) $.
\end{theoB}
\begin{theoC}(\cite[Theorem 1.4]{XU-CAO-MJM-2018})
	Let $ c\in\mathbb{C}^n\setminus\{0\} $. Then any non-constant entire solution with finite order of the equation
	\begin{align*}
		f^2(z)+f^2(z+c)=1
	\end{align*}
has the form $ f(z)=\cos\left(L(z)+B\right) $, where $ L $ is a linear function of the form $ L(z)=a_1z_1+\cdots+a_nc_n $ on $ \mathbb{C}^n $ such that $ L(c)=-\frac{\pi}{2}-2k\pi $, $ k $ is an integer, and $ B $ is a constant in $ \mathbb{C} $.
\end{theoC}
Motivated by the above discussion, and for further study in the topic, it is natural to raise the question.  
\begin{ques}\label{qn-1.1}
How to describe the entire solutions for the system of the Circular-type differential or differential-difference equations in $ \mathbb{C}^2$?
\end{ques}
Nevanlinna theory in one and several complex variables is a useful tool for the study of existence as well as growth properties of meromorphic solutions of different complex functional equations. In response to the Question \ref{qn-1.1}, solution of the following system equations
\begin{align*}
	\begin{cases}
		f_1(z_1, z_2)^2+f_2(z_1+c_1, z_2+c_2)^2=1,\vspace{1.2mm}\\
		f_2(z_1, z_2)^2+f_1(z_1+c_1, z_2+c_2)^2=1
	\end{cases}
\end{align*}
and 
\begin{align*}
	\begin{cases}
		\left(\dfrac{\partial f_1(z_1, z_2)}{\partial z_1} \right)^2+[f_2(z_1+c_1, z_2+c_2)-f_1(z_1, z_2)]^2=1,\vspace{1.2mm}\\
		\left(\dfrac{\partial f_2(z_1, z_2)}{\partial z_1} \right)^2+[f_1(z_1+c_1, z_2+c_2)-f_2(z_1, z_2)]^2=1
	\end{cases}
\end{align*}
are obtained in \cite{Xu-Liu-Li-JMAA-2020} explicitly using Nevanlinna theory in several complex variables. 
 
\subsection{Background of quadratic trinomial equations and their solutions in one and several complex variables:}
Utilizing the difference analogues of Nevanlinna theory in $\mathbb{C}$ by Halburd and Korhonen \cite{Halburd & Korhonen & 2006}, Chiang and Feng \cite{Chiang & Feng & 2008}, and subsequently in $\mathbb{C}^n$, respectively, many scholars have tried to characterize all the meromorphic solutions of Circular-type difference or differential-difference equations (see \cite{Ahamed-Allu-AMP-2023, Zheng-Xu & Analysis math & 2021, Xu-RMJM-2021, XU-CAO-MJM-2018, Ahamed-Mandal-JKMS-2024, Mandal-Ahamed-COAT-2024}). However, Saleeby (see \cite{Saleeby-AM-2013}) made some elementary observation on right factors of meromorphic function to describe complex analytic solutions to the quadratic trinomial functional equation $ f^2(z)+2\alpha f(z)g(z) +g^2(z) =1 $, $ \alpha^2\neq1$ constant in $ \mathbb{C} $ and obtain a result associate with the partial differential equations $ u^2_x+2\alpha u_x u_y +u^2_y=1 $. Later, in \cite{Liu & Yang & 2016}, Liu and Yang studied the existence and form of solutions of some quadratic trinomial functional equations and proved that if $ \alpha\neq \pm 1,0 $, then quadratic trinomial differential equation $ f(z)^2+2\alpha f(z)f^{\prime}(z)+f^{\prime}(z)^2=1 $ in $ \mathbb{C} $ has no transcendental meromorphic solutions, whereas the difference functional equation $ f(z)^2 +2\alpha f(z)f(z+c)+f(z+c)^2=1 $ in $ \mathbb{C} $ must have transcendental entire solutions of order equal to one. Recently, Xu and Jiang \cite{XU-JIANG-RACSAM-2022} obtained the forms of solutions for several systems of quadratic trinomial functional equations related to the above system in $\mathbb{C}^2$ are following
\begin{align*}
	\begin{cases}
		f_1(z_1, z_2)^2+2\alpha f_1(z_1, z_2)f_2(z_1+c_1, z_2+c_2)+f_2(z_1+c_1, z_2+c_2)^2=1,\vspace{1.2mm}\\
		f_2(z_1, z_2)^2+2\alpha f_2(z_1, z_2)f_1(z_1+c_1, z_2+c_2)+f_1(z_1+c_1, z_2+c_2)^2=1
	\end{cases}
\end{align*}
and 
\begin{align*}
	\begin{cases}
		f_1(z_1+c_1, z_2+c_2)^2 +2\alpha f_1(z_1+c_1, z_2+c_2)\dfrac{\partial f_2(z_1, z_2)}{\partial z_1} + \left(\dfrac{\partial f_2(z_1, z_2)}{\partial z_1} \right)^2=1, \vspace{1.2mm}\\
		f_2(z_1+c_1, z_2+c_2)^2 +2\alpha f_2(z_1+c_1, z_2+c_2)\dfrac{\partial f_1(z_1, z_2)}{\partial z_1} + \left(\dfrac{\partial f_1(z_1, z_2)}{\partial z_1} \right)^2=1,
	\end{cases}
\end{align*}
where $\alpha^2\neq 0,1$; $\alpha\in\mathbb{C}$. However, there have been several studies that focus on finding solutions of single equations or systems of equations (see \cite{Gao-JMAA-2016, Gauthier-AM-1972, Gundersen-PEMS-2020, Saleeby-JMAA-2015, Ishizaki-Korhonen-CA-2018, Korhonen-Zhang-CA-2020,Ahamed-Mandal-DM-2025}).

\subsection{System of quadratic equations and their solutions in several complex variables:}
The system of general quadratic equations has been an interesting subject in the field of complex analysis in connection with the application of Nevanlinna theory. It appears that the study of systems of the general quadratic equations which is a general setting of circular-type as well as quadratic trinomial-type functional equations in several complex variables has not been addressed in the literature before. In this article,  using Nevanlinna theory in several complex variables, our primary objective is to explore the existence as well as form of the solutions with finite order of general quadratic functional equations in $ \mathbb{C}^n $. \vspace{1.2mm}

Since a non-constant polynomials in $\mathbb{C}^n$ may exhibit periodic behavior, and we know that polynomials play a crucial role in the solutions of partial differential or difference equations in $\mathbb{C}^n$, the characterization of solutions in $\mathbb{C}^n$ will naturally differ from that in $\mathbb{C}$.  Further, we observe that the Circular equation $ f^2+g^2=1 $ and quadratic trinomial functional equation $ f^2+2\alpha fg+g^2=1 $ are both particular cases of general quadratic functional equation of the form
\begin{align*}
   \mathcal{A}(f, g):=af^2+2\alpha fg+bg^2+2\beta f+2\gamma g+C=0
\end{align*}
in $ \mathbb{C}^n $. The most general form of these equations is represented by quadratic equations with arbitrary coefficients. Therefore, the solutions for systems of general quadratic equations provide a broader perspective compared to solutions of these circular and quadratic trinomial equations in $\mathbb{C}^n$. \vspace{2mm}

We would like to highlight that while there has been extensive research on Fermat-type equations of circular or quadratic trinomial type exploring their solutions in one and several complex variables, the corresponding equations for general quadratic equations (GQE) have not received as much attention from researchers to date. Consequently, the nature of solutions for systems of general quadratic equations remains unknown. This lack of attention serves as the main motivation for our current article. Specifically, our focus is on finding solutions for systems of GQE in se veral complex variables, which also allows us to handle single GQE in both one and several complex variables. Therefore, our primary objective is to address this gap in the existing literature and contribute to the understanding of solutions for systems of GQE in $\mathbb{C}^n$.\vspace{2mm}

The following questions naturally arise for further study of quadratic trinomial functional equations in $ \mathbb{C}^n \;(n\geq 2)$, and seeking possible answers to these questions will be our main objective of the article. Although many equations with different settings can be investigated for their solutions, in this article, we consider three types of general quadratic functional equations in $\mathbb{C}^n$ as follows:
\begin{align*}
	\begin{cases}
		\mathcal{A}\left(f_1(z),\; f_2(z+c)\right)=0\\
		\mathcal{A}\left(f_2(z),\; f_1(z+c)\right)=0,\;
	\end{cases} 
\end{align*} 
\begin{align*}
	\begin{cases}
		\mathcal{A}\left(f_1(z),\; \dfrac{\partial f_2(z)}{\partial z_1}\right)=0\vspace{2mm}\\
		\mathcal{A}\left(f_2(z),\; \dfrac{\partial f_1(z)}{\partial z_1}\right)=0
	\end{cases} \;\;\;\; \mbox{and} \;\;\;\;\;\;\;\;
	\begin{cases}
		\mathcal{A}\left(f_1(z+c),\; \dfrac{\partial f_2(z)}{\partial z_1}\right)=0\vspace{2mm}\\
		\mathcal{A}\left(f_2(z+c),\; \dfrac{\partial f_1(z)}{\partial z_1}\right)=0.
	\end{cases}
\end{align*}
The above discussions thus lead us to raise a natural question.
\begin{ques}\label{q-1.1}
What can be said about the existence and forms of the solutions of the above system of general quadratic functional equations in $\mathbb{C}^n$, for $n\geq 2$?
\end{ques}

Motivated by the question above, this article focuses on describing transcendental solutions for various general quadratic difference equations and partial differential-difference equations in $\mathbb{C}^n$. This article aims to explore the existence and precise form of solutions for the mentioned system of general quadratic functional equations, utilizing Nevanlinna theory and the difference analog of Nevanlinna theory in $\mathbb{C}^n$ (see \cite{Cao-Korhonen-JMAA-2016, Korhonen-CMFT-2012,Saleeby-JMAA-2015}). These theories will provide a comprehensive answer to Question \ref{q-1.1}. Henceforth, throughout the article, we assume that the corresponding companion matrices for the equations are non-singular, \textit{i.e.,} $\Delta=:abC+2\alpha\beta\gamma-a\gamma^2 -b\beta^2 -C\alpha^2 \neq 0$. Otherwise, each equation in the system can be linearly factorized, which will eventually reveal that the solution is a constant. This suggests that further investigation is unnecessary. \vspace{2mm}

\subsection{Observations:} In $\mathbb{C}^2$, we see that if $p(z +c) - p(z)= \xi $, where $\xi$ is a constant in $\mathbb{C}$. Then, it follows that $p(z)= L(z) + H(s) + A_3$, where $L(z)= a_1 z_1 + a_2 z_2 $ and $H(s)$ is a polynomial in $s:= c_1 z_2 -c_2 z_1$ and $\xi=a_1 c_1 + a_2 c_2$. But, in $\mathbb{C}^3$, let $p(z)=A_1z_1+A_2z_2+A_3z_3+ (c_1z_2-c_2z_1)^2 +(c_2z_3-c_3z_2)^2$, where $c_1c_2c_3\neq 0$. Obviously, $p(z+c)- p(z)=\xi$, where $\xi=A_1c_1+A_2c_2+A_3c_3$. But, $p(z)$ cannot be represented as $p(z) = L(z) + H(s) + \xi$. Otherwise, if $A_1z_1+ A_2z_2+A_3z_3+(c_1z_2-c_2z_1)^2+(c_2z_3-c_3z_2)^2=L(z) + H(s_1) +\xi$ then $H(s_1)=B_2(\beta_1z_1+\beta_2z_2+\beta_3z_3)^2+B_1(\beta_1z_1+ \beta_2z_2+\beta_3z_3)+B_0$. By comparing the coefficients of $z_1z_2,z_2z_3,z_1z_3$ of both sides, it follows that
\begin{align}\label{1-4}
	2 B_2 \beta_1\beta_2=-2c_1c_2
\end{align}
\begin{align*}
	2B_2\beta_2\beta_3=-2c_2c_3
\end{align*}
\begin{align}\label{1-6}
	2B_2 \beta_1\beta_3=0.
\end{align}
Obviously, $B_2\neq0$. By \eqref{1-6}, it yields that $\beta_1=0$ or $\beta_3=0$. If $\beta_1=0$, then it follows from \eqref{1-4} that $c_1=0$ or $c_2=0$, this is a contradiction with the assumption of $c_1c_2c_3\neq0$. If $\beta_3=0$, this is also a contradiction with the assumption. Therefore, investigation of solution for $\mathbb{C}^n \; (n\geq 2)$ are very significant as we see that the characterization of the solutions are different. \vspace{2mm}

The structure of the article is as follows: In Section 2, we provide a comprehensive review of essential lemmas and notations. In Section 3, we present the main results along with their proofs. These results correspond to Question 2, and we include illustrative examples for each to explore the precise nature of transcendental entire solutions with finite order. Additionally, we offer several remarks that support our improvements and generalizations of the results.

\section{\bf Some notations and Key lemmas}
This section introduces the notations and key lemmas that will be utilized throughout the manuscript. Let $ z+c=(z_1+c_1,\ldots,z_n+c_n) $ and $ tz=(tz_1,\ldots,tz_n) $ for any $ z=(z_1,\ldots,z_n), c=(c_1,\ldots,c_n)\in\mathbb{C}^n $ and $ t\in\mathbb{C} $. Let us fix several notations which will be used subsequently: 
\begin{align*}
	\Delta:=\begin{vmatrix}
		a & \alpha & \beta\\
		\alpha & b & \gamma\\
		\beta & \gamma & C
	\end{vmatrix}=abC+2\alpha\beta\gamma-a\gamma^2-b\beta^2-C\alpha^2\;(\neq 0),\; D:=\begin{vmatrix}
	a & \alpha\\
	\alpha & b
\end{vmatrix}=ab-\alpha^2\;(\neq 0),
\end{align*}
\begin{align*}
	\begin{cases}
		\xi^{\pm}_1:=\dfrac{2\alpha}{\sqrt{\left((b-a)\pm\sqrt{(a-b)^2 +4\alpha^2}\right)^2 +4\alpha^2}},\vspace{1.2mm}\\ \eta^{\pm}_1:=\dfrac{(b-a)\pm\sqrt{(a-b)^2 +4\alpha^2}}{\sqrt{\left((b-a)\pm\sqrt{(a-b)^2 +4\alpha^2}\right)^2 +4\alpha^2}}.
	\end{cases}
\end{align*}
Also, we denote
\begin{align*}
   A^{\pm}:=\dfrac{(a+b)\pm\sqrt{(a-b)^2 +4\alpha^2}}{2},\;\;\;\; B^{\mp}:=\dfrac{(a+b)\mp\sqrt{(a-b)^2 +4\alpha^2}}{2}
\end{align*}
and
\begin{align*}
	\begin{cases}
		D_{11}:=\dfrac{\xi^{\pm}_1}{\sqrt{\frac{DA^{\pm}}{-\Delta}}},\;\; D_{12}:=\dfrac{\eta^{\pm}_1}{\sqrt{\frac{DB^{\mp}}{-\Delta}}}\;\;\mbox{and}\;\; T_1:=\dfrac{\alpha\gamma-b\beta}{ab-\alpha^2},\vspace{1.5mm}\\ E_{11}:=\dfrac{\eta^{\pm}_1}{\sqrt{\frac{DA^{\pm}}{-\Delta}}},\;\; E_{12}:=\dfrac{\xi^{\pm}_1}{\sqrt{\frac{DB^{\mp}}{-\Delta}}}\;\;\mbox{and}\;\; T_2:=\dfrac{\alpha\beta-a\gamma}{ab-\alpha^2}.
	\end{cases}
\end{align*}
Furthermore, for some consequences of the main results, we introduce some notations:
\begin{align*}
	\begin{cases}
		K_{11}:=\sqrt{(a+b)\pm\sqrt{(a-b)^2 + 4\alpha^2}},\;\; K_{12}:=\sqrt{\left((b-a)\pm\sqrt{(a-b)^2+4\alpha^2}\right)^2 +4\alpha^2},\vspace{1.5mm}\\ K_{13}:=\sqrt{(a+b)\mp\sqrt{(a-b)^2 + 4\alpha^2}}, \;\; K_{14}:=(b-a)\pm\sqrt{(a-b)^2+4\alpha^2)},\vspace{1.5mm}\\ A_{11}:=\dfrac{2\sqrt{2}\alpha}{K_{11}K_{12}},\;\;\;\;\; A_{12}:=\dfrac{\sqrt{2}K_{14}}{K_{12}K_{13}},\;\;\;\; B_{11}:=\dfrac{\sqrt{2}K_{14}}{K_{11}K_{12}},\;\;\;\; B_{12}:=\dfrac{2\sqrt{2}\alpha}{K_{12}K_{13}}.
	\end{cases}
\end{align*}\vspace{2mm}

In this study, Nevanlinna's theory for meromorphic functions plays a significant role in ensuring the existence and form of solutions. We assume familiarity with the basic  Nevanlinna  theory for meromorphic functions in several complex variables (see \cite{Yang-Yi-2006, Yang-SSP-1993,Hu-Li-Yang-2003} for details). \vspace{2mm}

For $z=(z_1,\ldots,z_n)\in\mathbb{C}^n$, set $||z||^2=\sum_{j=1}^{n} |z_j|^2$. For $r>0$, define 
\begin{align*}
	B_n(r)=\{z\in\mathbb{C}^n : |z|< r\}, \hspace{0.5cm} S_n(r)=\{z\in\mathbb{C}^n : |z|=r\}.
\end{align*}
Let $d=\partial +\overline{\partial}$, $d^c=(4\pi\sqrt{-1})^{-1} (\partial -\overline{\partial})$, then $dd^c=\frac{\sqrt{-1}}{2\pi} \partial\overline{\partial}$. Write
\begin{align*}
	\sigma_n(z):=(dd^c||z||^2)^{n-1}, \hspace{0.5cm} \eta_n(z):=d^c\log||z||^2\wedge (dd^c\log||z||^2)^{n-1}
\end{align*}
for $z\in\mathbb{C}^n\setminus\{0\}$. For a meromorphic function $f$ on
$\mathbb{C}^n$, $a\in\mathbb{C}\cup\{\infty\}$, let $\nu^{0}_{f-a}$ be the zero divisor of $f-a$. Set 
\begin{align*}
	\nu(r)=\sup \nu^{0}_{f-a} \cap B_n(t), \hspace{0.5cm} \nu^{[M]}(r) :=\min\{\nu(r),M\}
\end{align*}
\begin{align*}
	n_f\left(t,\frac{1}{f-a}\right)=\begin{cases}
		\int_{\nu(r)} \nu^{0}_{f-a} \sigma_n(z),\hspace{0.5cm} \mbox{if}\hspace{0.5cm} n\geq 2,
		\vspace{2mm}\\ \sum_{|z|\leq t} \nu^{0}_{f-a}(z), \hspace{0.5cm} \mbox{if}\hspace{0.5cm} n=1.
	\end{cases}
\end{align*}
\begin{align*}
	n^{[M]}_f\left(t,\frac{1}{f-a}\right)=\begin{cases}
		\int_{\nu^{[M]}(r)} \nu^{0}_{f-a} \sigma_n(z),\hspace{1.2cm} \mbox{if}\hspace{0.5cm} n\geq 2,
		\vspace{2mm}\\ \sum_{|z|\leq t} \min\{\nu^{0}_{f-a},M\}, \hspace{0.5cm} \mbox{if}\hspace{0.5cm} n=1.
	\end{cases}
\end{align*}
Whenever $M=1$, we usually use the notation $\overline{n_f}\left(r, \frac{1}{f-a}\right)$ instead of $n^{[M]}_f\left(t,\frac{1}{f-a} \right)$. We define the counting function and reduced counting function by
\begin{align*}
	N\left(r,\frac{1}{f-a}\right)=\int_{0}^{r} \frac{n_f\left(t,\frac{1}{f-a}\right)}{t^{2n-1}} dt
\end{align*}
and
\begin{align*}
	\overline{N}\left(r,\frac{1}{f-a}\right)=\int_{0}^{r} \frac{\overline{n_f}\left(t,\frac{1}{f-a}\right)}{t^{2n-1}} dt
\end{align*}
respectively. Now, we define the proximity function by
\begin{align*}
	m\left(r,\frac{1}{f-a}\right)=\int_{S_n(r)}\log^{+}\frac{1}{|f(z)-a|}\eta_n(z),\hspace{0.5cm}\mbox{if}
	\hspace{0.5cm} a\neq\infty
\end{align*}
and
\begin{align*}
	m(r,f)=\int_{S_n(r)}\log^{+}|f(z)|\eta_n(z),\hspace{0.5cm}\mbox{if}
	\hspace{0.5cm} a=\infty.
\end{align*}
The Nevanlinna characteristic function of f is defined by
\begin{align*}
	T(r,f)=m(r,f)+N(r,f).
\end{align*}
Define the order of $ f $ by
\begin{align*}
	\rho(f):=\overline{\lim_{r\rightarrow\infty}}\dfrac{\log^{+}T(r,f)}{\log r}.
\end{align*}\vspace{2mm}

First, we recall here some necessary lemmas which will play key roles in proving the main results of the article.
\begin{lemA}(\cite{Ronkin_AMS-1974,Stoll-AMS-1974})
For any entire function $ F $ on $ \mathbb{C}^n $, $ F(0)\neq 0 $ and put $\rho(n_F)=\rho < \infty $, where $ \rho(n_F) $ denotes be the order of the counting function of zeros of $ F $. Then there exist a canonical function $ f_F $ and a function $ g_F \in\mathbb{C}^n $ such that $ F(z) =f_F (z)e^{g_F (z)} $. For the special case $ n = 1 $, $ f_F $ is the canonical product of Weierstrass.
\end{lemA}
\begin{lemB}(\cite{Polya-JLMS-1926})
If $  g $ and $ h $ are entire functions on the complex plane $ \mathbb{C} $ and $ g(h) $ is an entire function of finite order, then there are only two possible cases: either
\begin{enumerate}
	\item [(i)] the internal function $ h $ is a polynomial and the external function $ g $ is of finite order; or
	\item[(ii)] the internal function $ h $ is not a polynomial but a function of finite order, and the external function $ g $ is of zero order.
\end{enumerate}
\end{lemB}
As usual, the notation $ ``||\; P" $ means that the assertion $ P $ holds for all large $ r\in [0,\infty) $ outside a set with finite Lebesgue measure.
\begin{lemC}(\cite{Hu-Li-Yang-2003})
Let $ f_j(\not \equiv 0) $, $ j=1,2,3 $, be meromorphic functions on $ \mathbb{C}^m $ such that $ f_1 $ is not constant and $ f_1+f_2+f_3=1 $ such that 
\begin{align*}
	||\sum_{j=1}^{3}\left\{N_2\left(r,\frac{1}{f_j}\right)+2\overline{N}(r,f_j)\right\}<\lambda T(r,f_1) + O(\log^{+} T(r,f_1)),
\end{align*}
for all $ r $ outside possibly a set with finite logarithmic measure, where $ \lambda<1 $ is a positive number. Then either $ f_2=1\;\mbox{or}\; f_3=1 $.
\end{lemC}

\section{\bf Main results and its proof}
Before we state the main results of this article, let us set a periodic polynomial $\Psi(z)$ in $\mathbb{C}^n \;(n\geq 2)$ (see \cite[p. 3]{Hal-Ban-DM-2023}) is the following
\begin{align}\label{Eq-3.1}
	\Psi(z)=\sum_{i_1=1}^{^nC_2} H^2_{i_1}(s^2_{i_1}) +\sum_{i_2=1}^{^nC_3} H^3_{i_2}(s^3_{i_2}) +\cdots+ \sum_{i_{n-2}=1}^{^nC_{n-1}} H^{n-1}_{i_{n-2}}(s^{n-1}_{i_{n-2}}) + H^{n}_{n-1} (s^n_{n-1}),
\end{align}
where $H^2_{i_1}$ is a polynomial in $s^2_{i_1}:=d^2_{i_1 j_1}z_{j_1} +d^2_{i_1 j_2}z_{j_2}$ with $d^2_{i_1 j_1}c_{j_1} +d^2_{i_1 j_2}c_{j_2}=0$, $1\leq i_1\leq {^nC_2}$ and $1\leq j_1<j_2\leq n$; $H^3_{i_2}$ is a polynomial in $s^3_{i_2}:=d^3_{i_2 j_1}z_{j_1} +d^3_{i_2 j_2}z_{j_2} +d^3_{i_2 j_3}z_{j_3}$ with $d^3_{i_2 j_1}c_{j_1} +d^3_{i_2 j_2}c_{j_2} +d^3_{i_2 j_3}c_{j_3}=0$, $1\leq i_2\leq {^nC_3}$ and $1\leq j_1<j_2<j_3\leq n$; $\ldots$; $H^{n-1}_{i_{n-2}}$ is a polynomial in $s^{n-1}_{i_{n-2}}:= d^{n-1}_{i_{n-2} j_1}z_{j_1} +d^{n-1}_{i_{n-2} j_2}z_{j_2} +\cdots+d^{n-1}_{i_{n-2} j_{n-1}} z_{j_{n-1}}$ with $d^{n-1}_{i_{n-2} j_1}c_{j_1} +d^{n-1}_{i_{n-2} j_2}c_{j_2} +\cdots+d^{n-1}_{i_{n-2} j_{n-1}} c_{j_{n-1}}=0$, $1\leq i_{n-2}\leq {^nC_{n-1}}$ and $1\leq j_1<j_2<\cdots<j_{n-1}\leq n$; and $H^{n}_{n-1}$ is a polynomial in $s^n_{n-1}:= d^{n}_{i_{n-1} 1}z_{1} +d^{n}_{i_{n-1} 2}z_{2} +\cdots+d^{n}_{i_{n-1} n} z_{n}$ with $d^{n}_{i_{n-1} 1}c_{1} +d^{n}_{i_{n-1} 2}c_{2} +\cdots+d^{n}_{i_{n-1} n} c_{n}=0$, where for each $k$, the representation of $s^k_i$ in terms of the conditions of $j_1, j_2,\ldots,j_k$ is unique. \vspace{2mm}

For the convenience of the reader, we will present our three main results involving difference, partial differential, and partial differential-difference equations, along with their proofs and comprehensive details in each sub-section. This structured approach allows for a systematic presentation of our findings, ensuring that the key properties and characteristics of these solutions are clearly highlighted.

\subsection{\bf System of GQE type difference equations:}
The first result of the article is the following, which answers Question \ref{q-1.1} completely in case difference equations in $\mathbb{C}^n$.
\begin{thm}\label{th-2.1}
Let $ c=(c_1,\ldots,c_n)\in\mathbb{C}^n \setminus\{(0,\ldots,0)\}$ and $a, b, C, \alpha, \beta, \gamma\in\mathbb{C}$ such that $ab\neq 0$, $\Delta\neq 0$, and $\alpha^2\neq 0, ab$. Then the pair of finite order transcendental entire solutions in $\mathbb{C}^n$ for the system of general quadratic difference equations 
\begin{align}\label{eq-3.1}
	\begin{cases}
		\mathcal{A}\left(f_1(z),\; f_2(z+c)\right)=0,\\
		\mathcal{A}\left(f_2(z),\; f_1(z+c)\right)=0\;
	\end{cases}
\end{align}
must be one of the forms
\begin{enumerate}
	\item [(i)] $ f_1(z)=D_{11}\cos[\gamma(z)+b_1] -D_{12}\sin[\gamma(z)+b_1]+T_1, \vspace{2mm} \\ f_2(z)=D_{11}\cos[\gamma(z)+b_2] -D_{12}\sin[\gamma(z)+b_2]+T_1 $,
\end{enumerate}
where $\gamma(z)=L(z)+\Psi(z)$, $L(z)$ is a linear function of the form $ L(z)=\sum_{r=1}^{n} a_r z_r $ and $\Psi(z)$ is a polynomial defined in \eqref{Eq-3.1}, and $b_1,b_2\in\mathbb{C} $ satisfy
\begin{align*}
	T_1=T_2,\;\; e^{2iL(c)}\equiv\frac{R^2_{13}}{R^2_{11}} \;\;\mbox{and}\;\; e^{2i(b_1-b_2)}\equiv\frac{R_{13}}{R_{11}R_{12}},
\end{align*}
where
\begin{align*}
	R_{11}=\frac{iD_{11}-D_{12}}{iE_{11}-E_{12}},\;\; R_{12}=\frac{iD_{11}+D_{12}}{iE_{11}-E_{12}},\;R_{13}=\frac{iE_{11}+E_{12}}{iE_{11}-E_{12}}.
\end{align*}\vspace{2mm}

\begin{enumerate}
	\item [(ii)] $ f_1(z)=D_{11}\cos[\gamma(z)+b_1] -D_{12}\sin[\gamma(z)+b_1]+T_1, \vspace{2mm}\\ f_2(z)=D_{11}\cos[\gamma(z)+b_2] +D_{12}\sin[\gamma(z)+b_2]+T_1 $,
\end{enumerate}
where $ D_{11},D_{12},\gamma(z) $ are stated as in $ (i) $ and  $b_1, b_2\in\mathbb{C} $ satisfy
\begin{align*}
	T_1=T_2,\;\;e^{2iL(c)}\equiv\frac{R_{13}}{R_{11}R_{12}}\;\;\mbox{and}
	\;\; e^{2i(b_1-b_2)}\equiv\frac{1}{R^2_{11}}.
\end{align*}
\end{thm}
\begin{proof}[\bf Proof of Theorem \ref{th-2.1}]
Assume that $ (f_1, f_2) $ is a pair of finite order transcendental entire solution of the system $ \eqref{eq-3.1} $. Considering $f_1=F+T_1$ and $f_2=G+T_2$, the first equation of the system \eqref{eq-3.1} can be expressed to the following trinomial form
\begin{align}\label{eq-44.2}
	aF^2 + 2\alpha FG + bG^2 + \dfrac{\Delta}{D} = 0.
\end{align} 
Further, we consider that $ F=\mathcal{U}\xi^{\pm}_1 - \mathcal{V}\eta^{\pm}_1 $ and $ G=\mathcal{U}\eta^{\pm}_1+ \mathcal{V}\xi^{\pm}_1 $. By a tedious computation, it becomes evident that \eqref{eq-44.2} is simplified to the Circular-type equation
\begin{align*}
	\left(\sqrt{\frac{DA^{\pm}}{-\Delta}}\;\mathcal{U}\right)^2 + \left(\sqrt{\frac{DB^{\mp}}{-\Delta}}\;\mathcal{V}\right)^2 =1.
\end{align*}
Since $ f_1 $ and $ f_2 $ are finite order entire solutions, it is easy to see that the functions $ F $, $ G $, $ \mathcal{U} $ and $ \mathcal{V} $ are also entire with finite order in $ \mathbb{C}^n $. Thus, in view of Theorem A, Lemmas  A and B, it follows that there exists a non-constant polynomial $ h_1:\mathbb{C}^n \rightarrow \mathbb{C} $ such that $ \sqrt{\frac{DA^{\pm}}{-\Delta}}\;\mathcal{U} =\cos h_1(z) $ and $ \sqrt{\frac{DB^{\mp}}{-\Delta}}\;\mathcal{V} =\sin h_1(z) $. This turns out that $ f_1(z) $ and $ f_2(z+c) $ can be obtained as
\begin{align}\label{eq-3.2}
	f_1(z)=D_{11}\cos h_1(z) -D_{12}\sin h_1(z)+T_1 ,\\\label{eq-3.3}f_2(z+c)=E_{11}\cos h_1(z) +E_{12}\sin h_1(z)+T_2.
\end{align}
By the similar argument, in the case of the second equation of the system $ \eqref{eq-3.1} $, there exists a non-constant polynomial $ h_2(z) $ in $ \mathbb{C}^n $ such that $ f_2(z) $ and $ f_1(z+c) $ can be expressed as
\begin{align}\label{eq-3.4}
	f_2(z)=D_{11}\cos h_2(z) -D_{12}\sin h_2(z)+T_1 ,\\\label{eq-3.5}f_1(z+c)=E_{11}\cos h_2(z) +E_{12}\sin h_2(z)+T_2.
\end{align}
In view of $\eqref{eq-3.2}\;\mbox{and}\;\eqref{eq-3.5}$, a simple computation  shows that
\begin{align*}
	D_{11}\cos h_1(z+c)-D_{12}\sin h_1(z+c)+T_1=E_{11}\cos h_2(z)+E_{12}\sin h_2(z)+T_2.
\end{align*}
This equation can be written as
\begin{align}\label{eq-3.6}
	R_{11}e^{i(h_1(z+c)+h_2(z))} + R_{12}e^{i(h_2(z)-h_1(z+c))} - [R_{13}e^{ih_2(z)}+R_{14}]e^{ih_2(z)}=1,
\end{align}
where $R_{11}$, $R_{12}$, $R_{13}$, and $R_{14}$ are defined by
\begin{align}\label{eq-3.7}
	\begin{cases}
		R_{11}=\dfrac{iD_{11}-D_{12}}{iE_{11}-E_{12}},\;\; R_{12}=\dfrac{iD_{11}+D_{12}}{iE_{11}-E_{12}},\vspace{2mm}\\
		R_{13}=\dfrac{iE_{11}+E_{12}}{iE_{11}-E_{12}},\;\; R_{14}=\dfrac{2i(T_2-T_1)}{iE_{11}-E_{12}}.
	\end{cases}
\end{align}
Clearly, \eqref{eq-3.6} assume of the form $g_{11}(z) + g_{12}(z) + g_{13}(z) = 1$, where
\begin{align*}
	\begin{cases}
		g_{11}(z)= R_{11}e^{i(h_1(z+c)+h_2(z))},\vspace{2mm}\\ g_{12}(z)= R_{12}e^{i(h_2(z)-h_1(z+c))},\vspace{2mm}\\ g_{13}(z)= - \left(R_{13}e^{ih_2(z)}+R_{14}\right)e^{ih_2(z)}.
	\end{cases}
\end{align*}
It is easy to see that the function $g_{13}(z)= -\left(R_{13}e^{ih_2(z)}+ R_{14}\right)e^{ih_2(z)}$ is non-constant. Otherwise, 
\begin{align*}
	e^{ih_2(z)}={\left(-R_{14}\pm\sqrt{R^2_{14}-4K_1R_{13}}\right)}/{2R_{13}}
\end{align*} 
for arbitrary constants $K_1$, then $h_2$ must be a constant, leading to a contradiction. The functions $ g_{1j} $ $ (j=1, 2, 3) $ have neither zeros nor poles. We also observe that
\begin{align*}
	\overline{N}\left(r, R_{13}e^{2ih_2(z)}+ R_{14}e^{ih_2(z)}\right) =N_2\left(r, \frac{1}{R_{13}e^{2ih_2(z)}+R_{14}e^{ih_2(z)}}\right) =S(r,f),
\end{align*}
\begin{align*}
	\overline{N}\left(r, R_{11}e^{i(h_1(z+c)+h_2(z))}\right)=N_2\left(r, \frac{1}{R_{11}e^{i(h_1(z+c)+h_2(z))}}\right)=S(r,f)
\end{align*}
and
\begin{align*}
	\overline{N}\left(r, R_{12}e^{i(h_2(z)-h_1(z+c))}\right)=N_2\left(r, \frac{1}{R_{12}e^{i(h_2(z)-h_1(z+c))}}\right)=S(r,f).
\end{align*}
Consequently, we have the following estimates
\begin{align*}
	\sum_{j=1}^{3}\left\{N_2\left(r,\frac{1}{g_{1j}}\right)+2\overline{N}(r,g_{1j})\right\} = o(\log^{+} T(r,g_{1j})).
\end{align*} 
Because of $ \lambda T(r,g_{11})> 0 $, the strict inequality 
\begin{align*}
	||\sum_{j=1}^{3}\left\{N_2\left(r,\frac{1}{g_{1j}}\right)+2\overline{N}(r,g_{1j})\right\}<\lambda T(r,g_{11}) + O(\log^{+} T(r,g_{11}))
\end{align*} 
is satisfied for all $r$ outside possibly a set with finite logarithmic measure, where $\lambda<1$ is a positive number. Hence, by Lemma C, we have
\begin{align*}
	R_{12}e^{i(h_2(z)-h_1(z+c))}\equiv 1, \;\;\;\;\mbox{or}\;\;\;\;	R_{11}e^{i(h_1(z+c)+h_2(z))}\equiv 1.
\end{align*}
Similarly, using $\eqref{eq-3.3}\;\mbox{and}\;\eqref{eq-3.4} $, we see that
\begin{align}\label{eq-3.8}
	D_{11}\cos h_2(z+c)-D_{12}\sin h_2(z+c)+T_1=E_{11}\cos h_1(z)+E_{12}\sin h_1(z)+T_2,
\end{align}
which can be expressed as
\begin{align}\label{eq-3.9}
	R_{11}e^{i(h_2(z+c)+h_1(z))} + R_{12}e^{i(h_1(z)-h_2(z+c))} - \left(R_{13}e^{ih_1(z)}+R_{14}\right)e^{ih_1(z)}=1.
\end{align} 
The equation \eqref{eq-3.9} is the form of $ g_{21}(z) + g_{22}(z) + g_{23}(z) = 1 $, where
\begin{align*}
	\begin{cases}
		g_{21}(z)= R_{11}e^{i(h_2(z+c)+h_1(z))},\vspace{2mm}\\ g_{22}(z)= R_{12}e^{i(h_1(z)-h_2(z+c))},\vspace{2mm}\\ g_{23}(z)= - \left(R_{13}e^{ih_1(z)}+R_{14}\right)e^{ih_1(z)}.
	\end{cases}
\end{align*}
We see that $g_{23}= -\left(R_{13}e^{ih_1(z)}+R_{14}\right)e^{ih_1(z)}$ is non constant.  Otherwise $h_1(z)$ must be a constant, leading to a contradiction. The functions $ g_{2j} $ $ (j=1, 2, 3) $ have no zeros and poles. We also observe that
\begin{align*}
	\overline{N}\left(r, R_{13}e^{2ih_1(z)}+R_{14}e^{ih_1(z)}\right) =N_2\left(r, \frac{1}{R_{13} e^{2ih_1(z)}+R_{14}e^{ih_1(z)}}\right) =S(r,f),
\end{align*}
\begin{align*}
	\overline{N}\left(r, R_{11}e^{i(h_2(z+c)+h_1(z))}\right)=N_2\left(r, \frac{1}{R_{11}e^{i(h_2(z+c)+h_1(z))}}\right)=S(r,f)
\end{align*}
and
\begin{align*}
	\overline{N}\left(r, R_{12}e^{i(h_1(z)-h_2(z+c))}\right)=N_2\left(r, \frac{1}{R_{12}e^{i(h_1(z)-h_2(z+c))}}\right)=S(r,f).
\end{align*}
Consequently, we have the following estimates
\begin{align*}
	\sum_{j=1}^{3}\left\{N_2\left(r,\frac{1}{g_{2j}}\right)+2\overline{N}(r,g_{2j})\right\} = o(\log^{+} T(r,g_{2j})).
\end{align*} 
Because of $ \lambda T(r,g_{21})> 0 $, the strict inequality 
\begin{align*}
	||\sum_{j=1}^{3}\left\{N\left(r,\frac{1}{g_{2j}}\right)+2\overline{N}(r,g_{2j})\right\}<\lambda T(r,g_{21}) + O(\log^{+} T(r,g_{21}))
\end{align*} 
is satisfied for all $r$ outside possibly a set with finite logarithmic measure, where $\lambda<1$ is a positive number. Hence, by Lemma C, we have
\begin{align*}
	R_{12}e^{i(h_1(z)-h_2(z+c))}\equiv 1, \;\;\;\;\mbox{or}\;\;\;\; 	R_{11}e^{i(h_2(z+c)+h_1(z))}\equiv 1.
\end{align*} 
To complete the proof, it is enough to discuss the following four possible cases.\vspace{1.2mm}
	
\noindent{\bf Case I.} First, we suppose that
\begin{align}\label{eq-3.10}
	\begin{cases}
		R_{12}e^{i(h_2(z)-h_1(z+c))}\equiv1,\vspace{2mm}\\	R_{12}e^{i(h_1(z)-h_2(z+c))}\equiv1.
	\end{cases}
\end{align}
It follows from \eqref{eq-3.10} that $ h_2(z)-h_1(z+c)=\xi_1 \;\mbox{and}\; h_1(z)-h_2(z+c)=\xi_2 $, where $ \xi_1\;\mbox{and}\; \xi_2 $  are two constants in $ \mathbb{C} $. By combining, we have $ h_k(z)-h_k(z+2c)=\xi_1+\xi_2 $; $ k=1,2 $. Since $ h_1\;\mbox{and}\; h_2 $ are non-constant polynomials in $ \mathbb{C}^n $, we conclude that $ h_1(z)=\gamma(z)+b_1\;\mbox{and}\; h_2(z)=\gamma(z)+b_2  $, where $\gamma(z)=L(z)+\Psi(z) $, $L(z)$ is a linear function of the form $ L(z)=a_1z_1+\cdots+a_nz_n $ and $\Psi(z)$ is a polynomial defined in \eqref{Eq-3.1}, and $b_1,b_2\in\mathbb{C}$. In other words, $ \deg_s \Psi(z)\geq 2 $. Otherwise, $ L(z)+\Psi(z)+b_i $ is still a linear form of $ z_1,\ldots,z_n$. Therefore, the equation \eqref{eq-3.10} becomes
\begin{align}\label{eq-3.11}
	\begin{cases}
		R_{12} e^{i(-L(c)-b_1+b_2)} \equiv 1,\vspace{2mm}\\	R_{12} e^{i(-L(c)+b_1-b_2)} \equiv 1.
	\end{cases}
\end{align}
In view of \eqref{eq-3.6} and the first equation \eqref{eq-3.10}, a simple computation shows that
\begin{align*}
	R_{11}e^{ih_1(z+c)}-R_{13}e^{ih_2(z)}= R_{14},
\end{align*}
or,
\begin{align*}
	R_{11}e^{i(L(c)+b_1-b_2)}-R_{13}= R_{14}e^{-ih_2(z)}.
\end{align*}
Since the left-hand side is constant, the right-hand side must be constant as well. Hence either $R_{14}=0$, which yields $T_1=T_2$, or $h_2$ is constant, which contradicts our assumptions. Therefore $R_{14}=0$ and consequently $T_1=T_2$. Hence, we get
\begin{align}\label{Eq-3.14}
	R_{11}e^{i(L(c)+b_1-b_2)}= R_{13}.
\end{align}
Again using \eqref{eq-3.9} and the second equation \eqref{eq-3.10} we have
\begin{align*}
	R_{11}e^{ih_{2}(z+c)}-R_{13}e^{ih_{1}(z)}=R_{14}
\end{align*}
or,
\begin{align*}
	R_{11}e^{i(L(c)-b_1+b_2)}-R_{13}= R_{14}e^{-ih_1(z)}.
\end{align*}
which implies that $R_{14}$ must be $0$. Otherwise, $h_1$ must be a constant, which is a contradiction. Therefore $R_{14}=0$, implies that $T_1=T_2$. Thus, we get
\begin{align}\label{Eq-3.15}
	R_{11}e^{i(L(c)-b_1+b_2)}= R_{13}.
\end{align}
Clearly, it follows from \eqref{eq-3.11}, \eqref{Eq-3.14} and \eqref{Eq-3.15} that 
\begin{align}
	e^{2iL(c)}\equiv\frac{R^2_{13}}{R^2_{11}}\;\;\;\;\mbox{and}\;\;\;\; e^{2i(b_1-b_2)}\equiv\frac{R_{13}}{R_{11}R_{12}}.
\end{align} 
Hence, from $ \eqref{eq-3.2}$ and $\eqref{eq-3.4}$, we obtain explicit form of the solutions
\begin{align*}
	f_1(z)&=D_{11}\cos h_1(z) -D_{12}\sin h_1(z)+T_1,\\& =D_{11}\cos(\gamma(z)+b_1) -D_{12}\sin(\gamma(z)+b_1)+T_1
\end{align*}
and
\begin{align*}
	f_2(z)&=D_{11}\cos h_2(z) -D_{12}\sin h_2(z)+T_1,\\& =D_{11}\cos(\gamma(z)+b_2) -D_{12}\sin(\gamma(z)+b_2)+T_1.
\end{align*}
	
\noindent{\bf Case II.} Suppose that
\begin{align*}
	\begin{cases}
		R_{12}e^{i(h_2(z)-h_1(z+c))}\equiv1,\vspace{2mm}\\R_{11}e^{i(h_2(z+c)+h_1(z))}\equiv1.
	\end{cases}
\end{align*}
The above equations clarify that $ h_2(z)-h_1(z+c)\equiv\xi_1\;\mbox{and}\l\; h_2(z+c)+h_1(z)\equiv\xi_2$, with $ \xi_1\;\mbox{and}\; \xi_2 $ being two constants in $ \mathbb{C} $. Therefore, we obtain $ h_2(z)+h_2(z+2c)\equiv\xi_1+\xi_2 $, which contradicts the assumption that $ h_2(z) $ is a non-constant polynomial in $ \mathbb{C}^n $. \vspace{1.2mm}
	
\noindent{\bf Case III.} Suppose that
\begin{align*}
	\begin{cases}
		R_{11}e^{i(h_1(z+c)+h_2(z))}\equiv 1,\vspace{2mm}\\	R_{12}e^{i(h_1(z)-h_2(z+c))}\equiv 1.
	\end{cases}
\end{align*}
The above two equations lead to $ h_1(z+c)+h_2(z)\equiv\xi_1\;\mbox{and}\; h_1(z)-h_2(z+c)\equiv\xi_2 $, where $ \xi_1\;\mbox{and}\; \xi_2 $ denote two constants in $ \mathbb{C} $. Thus, we see that $ h_1(z)+h_1(z+2c)\equiv\xi_1+\xi_2 $, which contradicts that $ h_1(z) $ is a non-constant polynomial in $\mathbb{C}^n$.\vspace{1.2mm}
	
\noindent{\bf Case IV.} Let
\begin{align}\label{eq-3.13}
	\begin{cases}
		R_{11}e^{i(h_1(z+c)+h_2(z))}\equiv 1,\vspace{2mm}\\R_{11}e^{i(h_2(z+c)+h_1(z))}\equiv 1.
	\end{cases}
\end{align}
From $ \eqref{eq-3.13} $ it follows that $ h_1(z+c)+h_2(z)\equiv\xi_1 \;\mbox{and}\; h_2(z+c)+h_1(z)\equiv\xi_2 $, where $ \xi_1\;\mbox{and}\; \xi_2 $ are two constants in $ \mathbb{C} $. Thus $ h_1(z+2c)-h_1(z) \equiv\xi_1-\xi_2 \;\mbox{and}\; h_2(z+2c)-h_2(z)\equiv\xi_2-\xi_1 $. Since $ h_1\;\mbox{and}\; h_2 $ are non-constant polynomial in $ \mathbb{C}^2 $ and similar to the argument as in Case I, we can conclude that $ h_1(z)=\gamma(z)+b^{*}_1\; \mbox{and}\;h_2(z)=-\gamma(z)+b^{*}_2  $, where $\gamma(z)=L(z)+\Psi(z)$, $L(z)$ is a linear function of the form $ L(z)=a_1z_1+\cdots+a_nz_n $ and $\Psi(z)$ is a polynomial defined in \eqref{Eq-3.1}, and $b^{*}_1,b^{*}_2\in\mathbb{C} $. So, the equation \eqref{eq-3.13} becomes
\begin{align}\label{eq-3.14}
	\begin{cases}
		R_{11}e^{i(L(c)+b^{*}_1+b^{*}_2)}\equiv1,\vspace{2mm}\\	R_{11}e^{i(-L(c)+b^{*}_1+b^{*}_2)}\equiv1.
	\end{cases}
\end{align}
Considering \eqref{eq-3.6} and the first equation \eqref{eq-3.13}, it is easy to see that
\begin{align*}
	R_{12}e^{-ih_1(z+c)}-R_{13}e^{ih_2(z)}=R_{14},
\end{align*}
or,
\begin{align*}
	R_{12}e^{i(-L(c)-b^{*}_1-b^{*}_2)}-R_{13}= R_{14}e^{-ih_2(z)},
\end{align*}
which implies that $R_{14}$ must be $0$. Otherwise, $h_2$ must be a constant, which is a contradiction. Consequently $T_1=T_2$.  Hence, we get
\begin{align}\label{Eq-3.19}
	R_{12}e^{i(-L(c)-b^{*}_1-b^{*}_2)}=R_{13}.
\end{align}
Again using \eqref{eq-3.9} and the second equation \eqref{eq-3.13} we have
\begin{align*}
	R_{12}e^{-ih_{2}(z+c)}-R_{13}e^{ih_{1}(z)}=R_{14}
\end{align*}
or,
\begin{align*}
	R_{12}e^{i(L(c)-b^{*}_1-b^{*}_2)}-R_{13}= R_{14}e^{-ih_1(z)},
\end{align*}
which implies that $R_{14}$ must be $0$. Otherwise, $h_1$ must be a constant, which leads to a contradiction. Therefore $R_{14}=0$, implies that $T_1=T_2$.  Thus, we get
\begin{align}\label{Eq-3.20}
	R_{12}e^{i(L(c)-b^{*}_1-b^{*}_2)}= R_{13}.
\end{align}
Hence, it follows from \eqref{eq-3.14}, \eqref{Eq-3.19} and \eqref{Eq-3.20} that 
\begin{align}\label{eq-3.15}
	e^{2iL(c)}\equiv\frac{R_{13}}{R_{11}R_{12}}\;\;\;\;\mbox{and}\;\;\;\; e^{2i(b^{*}_1+b^{*}_2)}\equiv\frac{1}{R^2_{11}}.
\end{align}
In view of  \eqref{eq-3.2} and \eqref{eq-3.4}, we obtain the following
\begin{align*}
	f_1(z)&=D_{11}\cos h_1(z) -D_{12}\sin h_1(z)+T_1,\\& =D_{11}\cos(\gamma(z)+b^{*}_1) -D_{12}\sin(\gamma(z)+b^{*}_1)+T_1
\end{align*}
and
\begin{align*}
	f_2(z)&=D_{11}\cos h_2(z) -D_{12}\sin h_2(z)+T_1,\\& =D_{11}\cos(-\gamma(z)+b^{*}_2) -D_{12}\sin(-\gamma(z)+b^{*}_2)+T_1,\\& =D_{11}\cos(\gamma(z)-b^{*}_2) + D_{12}\sin(\gamma(z)-b^{*}_2)+T_1.
\end{align*}
Letting $ b_1=b^{*}_1\;\mbox{and}\;b_2=-b^{*}_2 $, the solution $(f_1, f_2)$ takes the form
\begin{align*}
	f_1(z)=D_{11}\cos(\gamma(z)+b_1) -D_{12}\sin(\gamma(z)+b_1)+T_1,\\ f_2(z)=D_{11}\cos(\gamma(z)+b_2) + D_{12}\sin(\gamma(z)+b_2)+T_1.
\end{align*} 
That completes the proof.
\end{proof}

\begin{rem}
The following examples shows that equation \eqref{eq-3.1} admits transcendental entire solution with growth order $\geq 1$. This represents a significant difference between the solutions in $\mathbb{C}$ and those in $\mathbb{C}^n$. It is clear that Theorem \ref{th-2.1} improves the results from \cite[Theorem 1.2]{Xu-Liu-Li-JMAA-2020}, \cite[Theorem 2.1]{XU-JIANG-RACSAM-2022}, and \cite[Theorem 1.4]{Liu & Yang & 2016} for system of general quadratic differential-difference functional equations in $\mathbb{C}^n$.
\end{rem}

The examples given below clearly demonstrate the existence of transcendental entire solutions to the equation system in Theorem \ref{th-2.1}.
\begin{exm}
For $c=(3, 1, -2)\in\mathbb{C}^3$ and $ b_1=\left(1+\ln 5\right)/(2i)$, and $ b_2=1/(2i) $, by a routine computation, it can be easily
shown that the pair $(f_1,f_2)$, where
\begin{align*}
	f_1(z)&=\frac{-2}{\sqrt{10\pm 6\sqrt{5}}} \cos\left(2z_1+\frac{\ln 3}{2i} z_2 +3z_3 +(2z_1 -4z_2+z_3)^5  +\frac{\left(1+\ln 5\right)}{2i}\right) \\&\quad- \frac{(-1 \pm\sqrt{5})}{\sqrt{30\mp 14\sqrt{5}}} \sin \left(2z_1+\frac{\ln 3}{2i} z_2 +3z_3 +(2z_1 -4z_2+z_3)^2  +\frac{\left(1+\ln 5\right)}{2i}\right) - 2,
\end{align*}
and
\begin{align*}
	f_2(z)&=\frac{-2}{\sqrt{10\pm 6\sqrt{5}}} \cos\left(2z_1+\frac{\ln 3}{2i} z_2 +3z_3 +(2z_1 -4z_2+z_3)^5  +\frac{1}{2i}\right) \\&\quad+ \frac{(-1 \pm\sqrt{5})}{\sqrt{30\mp 14\sqrt{5}}} \sin \left(2z_1+\frac{\ln 3}{2i} z_2 +3z_3 +(2z_1 -4z_2+z_3)^4  +\frac{1}{2i}\right) - 2,
\end{align*}
is transcendental entire solutions with $\rho(f_1,f_2)=5>1$ in $\mathbb{C}^3$ of the equations
\begin{align*}
	\begin{cases}
		3f_1(z)^2 -4 f_1(z)f_2(z+c)+f_2(z+c)^2+ 4 f_1(z) -4 f_2(z+c)=1,\\ 3f_2(z)^2 -4 f_2(z)f_1(z+c)+f_1(z+c)^2+ 4 f_2(z) -4 f_1(z+c)=1.
	\end{cases}
\end{align*}
\end{exm}

\begin{rem}
Theorem \ref{th-2.1} gives rise to the following result as an immediate corollary that extends \cite[Theorem 1.2] {Xu-Liu-Li-JMAA-2020} and \cite[Theorem 2.1]{XU-JIANG-RACSAM-2022} to include arbitrary constants as coefficients of the equations, and also generalize from $\mathbb{C}^2$ to $\mathbb{C}^n$ for $n\geq 2$.
\end{rem}
\begin{cor}\label{cor-2.1}
Let $  c=(c_1,\ldots,c_n)\in\mathbb{C}^n\setminus\{(0, \ldots,0)\} $ and $ a,b,\alpha\in\mathbb{C} $ with $ ab\neq 0 $ and $\alpha^2\neq 0, ab $. Then the pair of finite order transcendental entire solutions in $\mathbb{C}^n$ for the system of general quadratic difference equations 
\begin{align*}
		\begin{cases}
			af_1(z)^2+2\alpha f_1(z)f_2(z+c)+bf_2(z+c)^2=1,\\ af_2(z)^2+2\alpha f_2(z)f_1(z+c)+bf_1(z+c)^2=1,
		\end{cases}
\end{align*}
must be one of the forms
\begin{enumerate}
		\item [(i)] $ f_1(z)=A_{11}\cos[\gamma(z)+b_1] -A_{12}\sin[\gamma(z)+b_1],\vspace{2mm}\\
		f_2(z)=A_{11}\cos[\gamma(z)+b_2] -A_{12}\sin[\gamma(z)+b_2] $,
\end{enumerate}
where $\gamma(z)=L(z)+\Psi(z)$, $L(z)$ is a linear function of the form $ L(z)=\sum_{r=1}^{n} a_r z_r $ and $\Psi(z)$ is a polynomial defined in \eqref{Eq-3.1}, and $b_1,b_2\in\mathbb{C} $ satisfy
\begin{align*}
		e^{2iL(c)}\equiv\frac{R^2_{13}}{R^2_{11}}\;\;\;\;\mbox{and}\;\;\;\; e^{2i(b_1-b_2)}\equiv\frac{R_{13}}{R_{11}R_{12}},
\end{align*}
where
\begin{align*}
		R_{11}=\frac{iA_{11}-A_{12}}{iB_{11}-B_{12}},\;\; R_{12}=\frac{iA_{11}+A_{12}}{iB_{11}-B_{12}}\;\;\mbox{and}\;\; R_{13}=\frac{iB_{11}+B_{12}}{iB_{11}-B_{12}}.
\end{align*}
\begin{enumerate}
		\item [(ii)] $ f_1(z)=A_{11}\cos[\gamma(z)+b_1] -A_{12}\sin[\gamma(z)+b_1],\vspace{2mm}\\f_2(z)=A_{11}\cos[\gamma(z)+b_2] +A_{12}\sin[\gamma(z)+b_2] $,
\end{enumerate}
	where $ A_{11}, A_{12},\gamma(z) $ are stated as in $ (i) $ and  $ b_1, b_2\in\mathbb{C} $ satisfy
	\begin{align*}
		e^{2iL(c)}\equiv\frac{R_{13}}{R_{11}R_{12}}\;\;\;\;\mbox{and}\;\;\;\; e^{2i(b_1-b_2)}\equiv\frac{1}{R^2_{11}},
	\end{align*}
	where $ R_{11},R_{12}\;\mbox{and}\;R_{13} $ are stated in $ (i) $ of the corollary \ref{cor-2.1}.
\end{cor}\vspace{2mm}

\subsection{\bf System of GQE type partial differential equations:}
The existence of solutions to the system of general quadratic partial differential equations is studied in the following result. In fact, this result improved \cite[Theorem 2.2]{XU-JIANG-RACSAM-2022} for more general sense.
\begin{thm}\label{th-2.2}
Let $ a,b,C,\alpha,\beta,\gamma$ in $\mathbb{C}$ such that $ab\neq 0$ and $\Delta\neq 0$, and $\alpha^2\neq 0, ab$. Then the system of partial differential equations 
 \begin{align}\label{eq-3.16}
 	\begin{cases}
 		\mathcal{A}\left(f_1(z),\; \dfrac{\partial f_2(z)}{\partial z_1}\right)=0,\vspace{2mm}\\
 		\mathcal{A}\left(f_2(z),\; \dfrac{\partial f_1(z)}{\partial z_1}\right)=0
 	\end{cases}
  \end{align}
has no pair of finite order transcendental entire solutions in $\mathbb{C}^n$.
\end{thm}
\begin{proof}[\bf Proof of Theorem \ref{th-2.2}]
We assume that $ (f_1,f_2) $ is a pair of finite order transcendental entire solutions of the system $ \eqref{eq-3.16} $. Considering $ f_1=F+ T_1 $ and $ \frac{\partial f_2}{\partial z_1}=G+T_2 $ in the first equation of the system \eqref{eq-3.16}, and then by using the analogous reasoning utilized in the proof of Theorem \ref{th-2.1}, it follows that there is a non-constant polynomial $ h_1:\mathbb{C}^n \rightarrow \mathbb{C} $ such that
\begin{align}\label{eq-3.17}
	f_1(z)=D_{11}\cos h_1(z) -D_{12}\sin h_1(z)+T_1 ,\\\label{eq-3.18}\dfrac{\partial f_2(z)}{\partial z_1}=E_{11}\cos h_1(z) +E_{12}\sin h_1(z)+T_2 .
\end{align}
Similarly, for the second equation of $ \eqref{eq-3.16} $, there exists a non-constant polynomial $ h_2(z) $ in $ \mathbb{C}^n $ such that
\begin{align}\label{eq-3.19}
	f_2(z)=D_{11}\cos h_2(z) -D_{12}\sin h_2(z)+T_1 ,\\\label{eq-3.20}\dfrac{\partial f_1(z)}{\partial z_1}=E_{11}\cos h_2(z) +E_{12}\sin h_2(z)+T_2.
\end{align}
Using $ \eqref{eq-3.17}\;\mbox{and}\;\eqref{eq-3.20} $, an easy computation shows that
\begin{align*}
	\left(-D_{11}\sin h_1(z)-D_{12}\cos h_1(z)\right)\dfrac{\partial h_1(z)}{\partial z_1}=E_{11}\cos h_2(z)+E_{12}\sin h_2(z)+T_2
\end{align*}
which turns out that
\begin{align}\label{eq-3.21}
	R_{21}\dfrac{\partial h_1(z)}{\partial z_1}e^{i(h_1(z)+h_2(z))}+ R_{22}\dfrac{\partial h_1(z)}{\partial z_1}e^{i(h_2(z)-h_1(z))} -(R_{23}e^{ih_2(z)}+R_{24})e^{ih_2(z)}=1,
\end{align}
where
\begin{align}\label{eq-3.22}
	\begin{cases}
		R_{21}=\dfrac{-D_{11}-iD_{12}}{iE_{11}-E_{12}},\;\; R_{22}=\dfrac{D_{11}-iD_{12}}{iE_{11}-E_{12}},\vspace{2mm}\\ R_{23}=\dfrac{iE_{11}+E_{12}}{iE_{11}-E_{12}},\;\; R_{24}=\dfrac{2iT_2}{iE_{11}-E_{12}}.
	\end{cases}
\end{align}
If $\frac{\partial h_1(z)}{\partial z_1}\equiv 0$, then it follows from $\eqref{eq-3.21}$ that $ -(R_{23}e^{ih_2(z)}+R_{24})e^{ih_2(z)}\equiv 1 $, which implies that $h_2$ must be a constant, which is a contradiction. Therefore, we must have $ \frac{\partial h_1(z)}{\partial z_1}\not\equiv 0 $. It is easy to see that equation \eqref{eq-3.21} is of the form $g_{31} (z) + g_{32}(z) + g_{33}(z) = 1$, where
\begin{align*}
	\begin{cases}
		g_{31}(z)= 	R_{21}\frac{\partial h_1(z)}{\partial z_1}e^{i(h_1(z)+h_2(z))},\vspace{2mm}\\ g_{32}(z)= R_{22}\frac{\partial h_1(z)}{\partial z_1}e^{i(h_2(z)-h_1(z))},\vspace{2mm}\\ g_{33}(z)= -(R_{23}e^{ih_2(z)}+R_{24})e^{ih_2(z)}.
	\end{cases}
\end{align*}
We also observe that
\begin{align*}
	\overline{N}\left(r, R_{23}e^{2ih_2(z)}+R_{24}e^{ih_2(z)}\right) =N_2\left(r, \frac{1}{R_{23} e^{2ih_2(z)}+R_{24}e^{ih_2(z)}}\right) =S(r,f),
\end{align*}
\begin{align*}
	\overline{N}\left(r, R_{21}\frac{\partial h_1(z)}{\partial z_1}e^{i(h_1(z)+h_2(z))}\right)=N_2\left(r, \frac{1}{R_{21}\frac{\partial h_1(z)}{\partial z_1}e^{i(h_1(z)+h_2(z))}}\right)=S(r,f)
\end{align*}
and
\begin{align*}
	\overline{N}\left(r, R_{22}\frac{\partial h_1(z)}{\partial z_1}e^{i(h_2(z)-h_1(z))}\right)=N_2\left(r, \frac{1}{R_{22}\frac{\partial h_1(z)}{\partial z_1}e^{i(h_2(z)-h_1(z))}}\right)=S(r,f).
\end{align*}
Because of $\lambda T(r,g_{31})> 0$, the strict inequality 
\begin{align*}
	||\sum_{j=1}^{3}\left\{N_2\left(r,\frac{1}{g_{3j}}\right)+2\overline{N}(r,g_{3j})\right\}<\lambda T(r,g_{31}) + O(\log^{+} T(r,g_{31}))
\end{align*} 
is satisfied for all $r$ outside possibly a set with finite logarithmic measure, where $\lambda<1$ is a positive number. Thus, using \eqref{eq-3.21}, in view of Lemma C, we obtain
\begin{align*}
	R_{22}\dfrac{\partial h_1(z)}{\partial z_1}e^{i(h_2(z)-h_1(z))}\equiv 1, \;\;\;\;\mbox{or}\;\;\;\; R_{21}\dfrac{\partial h_1(z)}{\partial z_1}e^{i(h_1(z)+h_2(z))}\equiv 1.
\end{align*}
Moreover, from \eqref{eq-3.18}\;\mbox{and}\;\eqref{eq-3.19}, an computation shows that
\begin{align*}
	\left(-D_{11}\sin h_2(z)-D_{12}\cos h_2(z)\right)\dfrac{\partial h_2(z)}{\partial z_1}=E_{11}\cos h_1(z)+E_{12}\sin h_1(z)+T_2,
\end{align*}
which can be written as 
\begin{align}\label{eq-3.23}
	R_{21}\dfrac{\partial h_2(z)}{\partial z_1}e^{i(h_2(z)+h_1(z))}+ R_{22}\dfrac{\partial h_2(z)}{\partial z_1}e^{i(h_1(z) -h_2(z))}-(R_{23}e^{ih_1(z)}+R_{24})e^{ih_1(z)}=1.
\end{align}
If $ \frac{\partial h_2(z)}{\partial z_1}\equiv 0 $, then it follows from $ \eqref{eq-3.23} $ that $ -(R_{23}e^{ih_1(z)}+R_{24})e^{ih_1(z)}\equiv 1 $, which implies that $h_1$ is a constant, a contradiction. Therefore, we must have $ \frac{\partial h_2(z)}{\partial z_1}\not\equiv 0 $. Now, we see that equation \eqref{eq-3.23} is of the form $g_{41}(z) + g_{42}(z) + g_{43}(z) = 1$, where
\begin{align*}
	\begin{cases}
		g_{41}(z)= 	R_{21}\frac{\partial h_2(z)}{\partial z_1}e^{i(h_2(z) +h_1(z))},\vspace{2mm}\\ g_{42}(z)= R_{22}\frac{\partial h_2(z)}{\partial z_1}e^{i(h_1(z) -h_2(z))},\vspace{2mm}\\ g_{43}(z)=-(R_{23}e^{ih_1(z)}+ R_{24})e^{ih_1(z)}.
	\end{cases}
\end{align*}
Moreover, we note that
\begin{align*}
	\overline{N}\left(r, R_{23}e^{2ih_1(z)}+R_{24}e^{ih_1(z)}\right) =N_2\left(r, \frac{1}{R_{23} e^{2ih_1(z)}+R_{24}e^{ih_1(z)}}\right) =S(r,f),
\end{align*}
\begin{align*}
	\overline{N}\left(r, R_{21}\frac{\partial h_2(z)}{\partial z_1}e^{i(h_2(z)+ h_1(z))}\right)=N_2\left(r, \frac{1}{R_{21}\frac{\partial h_2(z)}{\partial z_1}e^{i(h_2(z)+h_1(z))}}\right)=S(r,f)
\end{align*}
and
\begin{align*}
	\overline{N}\left(r, R_{22}\frac{\partial h_2(z)}{\partial z_1}e^{i(h_1(z) -h_2(z))}\right)=N_2\left(r, \frac{1}{R_{22}\frac{\partial h_2(z)}{\partial z_1}e^{i(h_1(z) -h_2(z))}}\right)=S(r,f).
\end{align*}
Because of $\lambda T(r,g_{41})> 0$, the strict inequality 
\begin{align*}
	||\sum_{j=1}^{3}\left\{N_2\left(r,\frac{1}{g_{4j}}\right)+2\overline{N}(r,g_{4j})\right\}<\lambda T(r,g_{41}) + O(\log^{+} T(r,g_{41}))
\end{align*} 
is satisfied for all $r$ outside possibly a set with finite logarithmic measure, where $\lambda<1$ is a positive number. By Lemma C, we easily obtain from \eqref{eq-3.23} that
\begin{align*}
	R_{22}\dfrac{\partial h_2(z)}{\partial z_1}e^{i(h_1(z)-h_2(z))}\equiv 1, \;\;\;\;\mbox{or}\;\;\;\; R_{21}\dfrac{\partial h_2(z)}{\partial z_1}e^{i(h_2(z)+h_1(z))}\equiv 1.
\end{align*}
To conclude the proof, we only need to consider the subsequent cases.\vspace{1.2mm}
	
\noindent{\bf Case A.} Suppose that
\begin{align}\label{eq-3.24}
	\begin{cases}
		R_{22}\dfrac{\partial h_1(z)}{\partial z_1}e^{i(h_2(z)-h_1(z))}\equiv 1 \vspace{2mm}\\
		R_{22}\dfrac{\partial h_2(z)}{\partial z_1}e^{i(h_1(z)-h_2(z))}\equiv 1
	\end{cases}
\end{align}
By using these two equations, we can easily deduce that $h_1(z) - h_2(z)$ is equal to a constant in $\mathbb{C}$, which we denote as $\lambda_1$. Hence, both $ \frac{\partial h_1(z)}{\partial z_1}\;\mbox{and}\; \frac{\partial h_2(z)}{\partial z_1} $ must be constants. Assume that $ \frac{\partial h_1(z)}{\partial z_1}\equiv c_0 $. Since $ h_1$ and $h_2 $ are non-constant polynomials in $ \mathbb{C}^n $, hence $ h_1(z)=c_0z_1+p(z_2,\ldots,z_n)+b_1 $, where $ p(z_2,\ldots,z_n) $ is a polynomial in $z_2,\ldots,z_n$ only in $ \mathbb{C}^{n} $ and $ b_1\in\mathbb{C} $. Since $ h_1(z)-h_2(z)=\lambda_1 $, we see that $ h_2(z)=c_0z_1+ p(z_2,\ldots,z_n)+b_2 $, where $ b_2\in\mathbb{C} $. Thus, the equation \eqref{eq-3.24} becomes
\begin{align}\label{eq-3.25}
	\begin{cases}
		R_{22}c_0e^{i(b_2-b_1)}\equiv 1,\vspace{2mm}\\ R_{22}c_0e^{i(b_1-b_2)}\equiv 1.
	\end{cases}
\end{align}
Using \eqref{eq-3.21} and the first equation of \eqref{eq-3.24}, we have
\begin{align*}
	R_{21}\frac{\partial h_1(z)}{\partial z_1}e^{ih_1(z)}-R_{23}e^{ih_2(z)}= R_{24},
\end{align*}
or,
\begin{align*}
	R_{21}c_0e^{i(b_1-b_2)}-R_{23}= R_{24}e^{-ih_2(z)},
\end{align*}
which implies that $R_{24}$ must be $0$. Otherwise, $h_2$ must be a constant, which is a contradiction. Therefore, we get $T_2=0$.  Hence, we get
\begin{align}\label{Eq-3.32}
	R_{21}c_0e^{i(b_1-b_2)}= R_{23}.
\end{align}
Again, in consideration of \eqref{eq-3.23} and the second equation of \eqref{eq-3.24}, it is easy to see that
\begin{align*}
	R_{21}\frac{\partial h_2(z)}{\partial z_1}e^{ih_2(z)}- R_{23}e^{ih_1(z)}= R_{24},
\end{align*}
or,
\begin{align*}
	R_{21}c_0e^{i(b_2-b_1)}-R_{23}= R_{24}e^{-ih_1(z)},
\end{align*}
which implies that $R_{24}$ must be $0$. Otherwise, $h_1$ must be a constant, which is a contradiction. Consequently $T_2=0$.  Hence, we get
\begin{align}\label{Eq-3.33}
	R_{21}c_0e^{i(b_2-b_1)}= R_{23}.
\end{align}
A simple computation using the equations in \eqref{eq-3.25}, \eqref{Eq-3.32} and \eqref{Eq-3.33} shows that
\begin{align*}
	c_0^2=\frac{R_{23}}{R_{21}R_{22}}\;\;\;\;\;\mbox{and}\;\;\;\; R_{22}R_{23}=R_{21}.
\end{align*}
Thus we have $ 2i\xi^{\pm}_1\eta^{\pm}_1(A^{\pm} - B^{\mp})=0 $ which implies that $ A^{\pm}=B^{\mp} $, hence
\begin{align*}
	+\sqrt{(a-b)^2+4\alpha^2}=-\sqrt{(a-b)^2+4\alpha^2}\;\;\mbox{or}\;\;-\sqrt{(a-b)^2+4\alpha^2}=+\sqrt{(a-b)^2+4\alpha^2}
\end{align*}
Therefore, both the equations yield that $ 2\sqrt{(a-b)^2+4\alpha^2}=0 $. Consequently, it follows that $ a=b $ and $ \alpha=0 $, which contradicts $ \alpha^2\neq ab $. \vspace{1.2mm}
	
\noindent{\bf Case B.} Let
\begin{align*}
	\begin{cases}
		R_{22}\dfrac{\partial h_1(z)}{\partial z_1}e^{i(h_2(z)-h_1(z))}\equiv 1,\vspace{2mm}\\  R_{21}\dfrac{\partial h_2(z)}{\partial z_1}e^{i(h_2(z)+h_1(z))}\equiv 1.	
	\end{cases}
\end{align*}
By the similar argument used in Case A, we see that $ h_2(z)-h_1(z)=\lambda_1 $ and $  h_2(z)+h_1(z)=\lambda_2  $, where $ \lambda_1 $ and $ \lambda_2 $ are two constants in $ \mathbb{C} $. Thus it follows that $ 2h_2(z)=\lambda_1+\lambda_2 $, a constant, which contradicts the fact that $ h_2(z) $ is non-constant.\vspace{1.2mm}
	
\noindent{\bf Case C.} Suppose that
\begin{align*}
	\begin{cases}
		R_{21}\dfrac{\partial h_1(z)}{\partial z_1}e^{i(h_1(z)+h_2(z))}\equiv 1,\vspace{2mm} \\R_{22}\dfrac{\partial h_2(z)}{\partial z_1}e^{i(h_1(z)-h_2(z))}\equiv 1.
	\end{cases}
\end{align*}
The argument used in Case B can be applied similarly to lead to a contradiction. \vspace{1.2mm}
	
\noindent{\bf Case D.} Let
\begin{align}\label{eq-3.26}
	\begin{cases}
		R_{21}\dfrac{\partial h_1(z)}{\partial z_1}e^{i(h_1(z)+h_2(z))}\equiv 1,\vspace{2mm}\\	R_{21}\dfrac{\partial h_2(z)}{\partial z_1}e^{i(h_2(z)+h_1(z))}\equiv 1.
	\end{cases}
\end{align}
In view of the equations in $ \eqref{eq-3.26} $, it follows that $ h_1(z)+h_2(z)\equiv \mbox{constant} $, and hence both $ \frac{\partial h_1(z)}{\partial z_1}\;\mbox{and}\; \frac{\partial h_2(z)}{\partial z_1} $ are constants in $ \mathbb{C} $. By using the same reasoning as in Case A, we can readily obtain $ h_1(z)=d_0z_1+q(z_2,\ldots,z_n)+b_1 $ and $ h_2(z)=-d_0z_1-q(z_2,\ldots,z_n)+b_2 $, where $ q(z_2,\ldots,z_n) $ is a polynomial of $z_2,\ldots,z_n$ only in $ \mathbb{C}^{n} $ and $ b_1,b_2\in\mathbb{C} $. Thus, it follows from $ \eqref{eq-3.26} $ that 
\begin{align*}
	\begin{cases}
		R_{21}re^{i(b_1+b_2)}\equiv 1,\vspace{2mm}\\-R_{21}re^{i(b_1+b_2)}\equiv 1,
	\end{cases}
\end{align*}
from which we obtain that $ -1\equiv 1 $, which is absurd. This completes the proof.
\end{proof}	
\begin{rem}
We see that the system of equations in Theorem \ref{th-2.2} does not admit a pair of finite transcendental entire solutions. However, by replacing $f_1(z)$ and $f_2(z)$ with $f_1(z+c)$ and $f_2(z+c)$ respectively, a pair of finite order transcendental entire solutions can be obtain for $c=(c_1,\ldots,c_n)\setminus\{(0,\ldots, 0)\}$. This can be seen in Theorem \ref{th-2.3} below.
\end{rem}\vspace{2mm}

\subsection{\bf System of GQE type partial differential-difference equations:}
In the next result, using Nevanlinna theory, we confirm the existence of solutions to a system of the general quadratic differential-difference functional equations in $ \mathbb{C}^n $ and obtain the precise form of those solutions, which answers Question \ref{q-1.1}.
\begin{thm}\label{th-2.3}
Let $ c=(c_1,\ldots,c_n)\in\mathbb{C}^n\setminus\{(0,\ldots, 0)\} $, $d_1\in\mathbb{C}\setminus\{0\}$ and $ a,b,C, \alpha, \beta, \gamma$ in $\mathbb{C} $ with $ab\neq 0$ and  $\Delta\neq 0$, and $\alpha^2\neq 0, ab$. Then any pair of finite order transcendental entire solutions in $\mathbb{C}^n$ for the system of partial differential-difference equations
\begin{align}\label{eq-3.27}
 	\begin{cases}
 		\mathcal{A}\left(f_1(z+c),\; \dfrac{\partial f_2(z)}{\partial z_1}\right)=0,\vspace{2mm}\\
 		\mathcal{A}\left(f_2(z+c),\; \dfrac{\partial f_1(z)}{\partial z_1}\right)=0
 	\end{cases}
\end{align}
must be one of the forms
\begin{enumerate}
	\item [(i)]$ f_1(z)=D_{11}\cos[L(z)-L(c)+b_1] -D_{12}\sin[L(z)-L(c)+b_1]+T_1 $,\vspace{2mm}\\ 
	 $ f_2(z)=D_{11}\cos[L(z)-L(c)+b_2] -D_{12}\sin[L(z)-L(c)+b_2]+T_1  $,
\end{enumerate}
where $ L(z)=\sum_{r=1}^{n} a_r z_r $ and $ b_1,b_2\in\mathbb{C} $ satisfy
\begin{align*}
	T_2=0,\;\;\;a^2_1\equiv\frac{R_{23}}{R_{21}R_{22}},\;\;\;e^{2i(b_1-b_2)}\equiv 1, \;\;\; e^{2iL(c)}\equiv\frac{R_{21}}{R_{22}R_{23}},
\end{align*}
for
\begin{align*}
	R_{21}=\frac{-D_{11}-iD_{12}}{iE_{11}-E_{12}},\;\; R_{22}=\frac{D_{11}-iD_{12}}{iE_{11}-E_{12}},\;\; R_{23}=\frac{iE_{11}+E_{12}}{iE_{11}-E_{12}}.
\end{align*}\vspace{2mm}

\begin{enumerate}
	\item [(ii)] $ f_1(z)=D_{11}\cos[L(z)-L(c)+b_1]- D_{12}\sin[L(z)-L(c)+b_1]+T_1,\vspace{2mm}\\   f_2(z) =D_{11}\cos[L(z)-L(c)+b_2] + D_{12}\sin[L(z)-L(c)+b_2] +T_1 $,
\end{enumerate}
 $ D_{11},D_{12} $ and $L(z)$ are stated as in $(i)$ and $ b_1,b_2 \in\mathbb{C} $ satisfy
\begin{align*}
	T_2=0,\;\;\; a^2_1\equiv \frac{R_{23}}{R_{21}R_{22}},\;\;\; e^{2i(b_1-b_2)}\equiv-\frac{R_{22}}{R_{21}R_{23}},\;\;\; e^{2iL(c)}\equiv -1,
\end{align*}
where $ R_{21},R_{22}\;\mbox{and}\;R_{23} $ are stated in $ (i) $.
\end{thm}
\begin{proof}[\bf Proof of Theorem \ref{th-2.3}]
Suppose that $ (f_1,f_2) $ is a pair of finite order transcendental entire solutions of the system $ \eqref{eq-3.27} $. We consider $ f_1= F+ T_1 $ and $ \frac{\partial f_2}{\partial z_1}= G+ T_2 $, in the first equation of the system \eqref{eq-3.27} and by the similar arguments being used in the proof of Theorem \ref{th-2.1}, it follows that there exists a non-constant polynomial $ h_1:\mathbb{C}^n\rightarrow \mathbb{C} $ such that
\begin{align}\label{eq-3.28}
	f_1(z+c)=D_{11}\cos h_1(z) -D_{12}\sin h_1(z)+T_1 ,\\\label{eq-3.29}\dfrac{\partial f_2(z)}{\partial z_1}=E_{11}\cos h_1(z) +E_{12}\sin h_1(z)+T_2 .
\end{align}
Similarly, for the second equation of the system $ \eqref{eq-3.27} $, there exists a non-constant polynomial $ h_2(z) $ in $ \mathbb{C}^n $ such that
\begin{align}\label{eq-3.30}
	f_2(z+c)=D_{11}\cos h_2(z) -D_{12}\sin h_2(z)+T_1 ,\\\label{eq-3.31}\dfrac{\partial f_1(z)}{\partial z_1}=E_{11}\cos h_2(z) +E_{12}\sin h_2(z)+T_2.
\end{align}
From $ \eqref{eq-3.28}\;\mbox{and}\;\eqref{eq-3.31} $ we can deduce that
\begin{align*}
	\left(-D_{11}\sin h_1(z)-D_{12}\cos h_1(z)\right)&\dfrac{\partial h_1(z)}{\partial z_1}\\&=E_{11}\cos h_2(z+c)+E_{12}\sin h_2(z+c)+T_2,
\end{align*}
which can be written as 
\begin{align}\label{eq-3.32}
	R_{21}\dfrac{\partial h_1(z)}{\partial z_1}e^{i(h_1(z)+h_2(z+c))}+ R_{22}\dfrac{\partial h_1(z)}{\partial z_1}& e^{i(h_2(z+c)-h_1(z))} \\& \nonumber -(R_{23}e^{ih_2(z+c)}+R_{24})e^{ih_2(z+c)}=1,
\end{align}
where
\begin{align}\label{eq-3.33}
	\begin{cases}
		R_{21}=\dfrac{-D_{11}-iD_{12}}{iE_{11}-E_{12}},\;\; R_{22}=\dfrac{D_{11}-iD_{12}}{iE_{11}-E_{12}},\vspace{2mm}\\ R_{23}=\dfrac{iE_{11}+E_{12}}{iE_{11}-E_{12}},\;\; R_{24}=\dfrac{2iT_2}{iE_{11}-E_{12}}.
	\end{cases}
\end{align}
Moreover, based on $ \eqref{eq-3.29}\;\mbox{and}\;\eqref{eq-3.30} $, it is evident that
\begin{align*}
	\left(-D_{11}\sin h_2(z)-D_{12}\cos h_2(z)\right)&\dfrac{\partial h_2(z)}{\partial z_1}\\&=E_{11}\cos h_1(z+c)+E_{12}\sin h_1(z+c)+T_2.
\end{align*}
It follows that
\begin{align}\label{eq-3.34}
	R_{21}\dfrac{\partial h_2(z)}{\partial z_1}e^{i(h_1(z+c)+h_2(z))}+ R_{22}\dfrac{\partial h_2(z)}{\partial z_1}&e^{i(h_1(z+c)-h_2(z))}\\& \nonumber -(R_{23}e^{ih_1(z+c)}+R_{24})e^{ih_1(z+c)}=1,
\end{align}
where $ R_{21},R_{22},R_{23},\;\mbox{and}\;R_{24} $ are non-zero constants as defined in $ \eqref{eq-3.33} $. We see that $ \frac{\partial h_1}{\partial z_1}\not\equiv 0 $ and $ \frac{\partial h_2}{\partial z_1}\not\equiv 0 $. Otherwise, from \eqref{eq-3.32} and \eqref{eq-3.34}, it is easy to see that $h_2$ and $h_1$ must be a constants. \vspace{2mm}

We further observe that
\begin{align*}
	&\overline{N}\left(r, R_{23}e^{2ih_2(z+c)}+R_{24}e^{ih_2(z+c)}\right) =N_2\left(r,\frac{1}{R_{23} e^{2ih_2(z+c)}+R_{24} e^{ih_2(z+c)}} \right) =S(r,f), \\& \overline{N}\left(r, R_{23}e^{2ih_1(z+c)} +R_{24}e^{ih_1(z+c)}\right)=N_2\left(r,\frac{1}{R_{23} e^{2ih_1(z+c)}+R_{24}e^{ih_1(z+c)}}\right)=S(r,f),
\end{align*}
\begin{align*}
	&\overline{N}\left(r, R_{21}\frac{\partial h_1(z)}{\partial z_1}e^{i(h_1(z) +h_2(z+c))}\right)=N_2\left(r, \frac{1}{R_{21}\frac{\partial h_1(z)}{\partial z_1}e^{i(h_1(z)+h_2(z+c))}}\right)=S(r,f), \\&
	\overline{N}\left(r, R_{21}\frac{\partial h_2(z)}{\partial z_1}e^{i(h_1(z+c) +h_2(z))}\right)=N_2\left(r, \frac{1}{R_{21}\frac{\partial h_2(z)} {\partial z_1}e^{i(h_1(z+c)+h_2(z))}}\right)=S(r,f)
\end{align*}
and
\begin{align*}
	&\overline{N}\left(r, R_{22}\frac{\partial h_1(z)}{\partial z_1} e^{i(h_2(z+c)-h_1(z))}\right)=N_2\left(r, \frac{1}{R_{22} \frac{\partial h_1(z)}{\partial z_1} e^{i(h_2(z+c)- h_1(z))}}\right)=S(r,f) , \\& \overline{N}\left(r, R_{22}\frac{\partial h_2(z)}{\partial z_1}e^{i(h_1(z+c)-h_2(z))}\right)=N_2\left(r, \frac{1}{R_{22}\frac{\partial h_2(z)}{\partial z_1}e^{i(h_1(z+c)-h_2(z))}}\right)=S(r,f).
\end{align*}
Thus, in view of Lemma C and it yields from \eqref{eq-3.32} and \eqref{eq-3.34} that
\begin{align*}
	R_{22}\dfrac{\partial h_1(z)}{\partial z_1}e^{i(h_2(z+c)-h_1(z))}\equiv 1, \;\;\;\;\mbox{or}\;\;\;\; R_{21}\dfrac{\partial h_1(z)}{\partial z_1}e^{i(h_1(z)+h_2(z+c))}\equiv 1
\end{align*}
and
\begin{align*}
	R_{22}\dfrac{\partial h_2(z)}{\partial z_1}e^{i(h_1(z+c)-h_2(z))}\equiv 1, \;\;\;\;\mbox{or}\;\;\;\; R_{21}\dfrac{\partial h_2(z)}{\partial z_1}e^{i(h_1(z+c)+h_2(z))}\equiv 1.
\end{align*}
To complete the proof, we need only to discuss the following possible cases.\vspace{1.2mm}
	
\noindent{\bf Case 1.} Let
\begin{align}\label{eq-3.35}
	\begin{cases}
		R_{22}\dfrac{\partial h_1(z)}{\partial z_1}e^{i(h_2(z+c)-h_1(z))} \equiv 1 \vspace{2mm}\\
		R_{22}\dfrac{\partial h_2(z)}{\partial z_1}e^{i(h_1(z+c)-h_2(z))} \equiv 1.
	\end{cases}
\end{align}
Hence, it follows that $h_2(z+c)-h_1(z)\equiv\xi_1$ and $h_1(z+c)-h_2(z) \equiv\xi_2 $, where $\xi_1$ and $\xi_2$ are two constant in $\mathbb{C} $. This implies that $h_k(z+2c)-h_k(z)\equiv \xi_1+\xi_2$ for $ k=1,2 $. We claim that $ h_1(z)=L(z)+b_1 $ and $ h_2(z)=L(z)+b_2 $, where $ L(z) $ is a linear function as the form $L(z)=a_1z_1+\cdots+a_nz_n$ and $ b_1, b_2$ are two constant in $\mathbb{C}$. In fact, since $h_1, h_2$ are two non-constant polynomials and  $ h_k(z+2c)-h_k(z)$ are two constants, we conclude that $h_1(z)=L(z)+\Psi(z)+b_1$ and $h_2(z)=L(z)+\Psi(z)+b_2$, where $\Psi(z)$ is a polynomial defined in \eqref{Eq-3.1}. Now from $ \eqref{eq-3.35}$ it is easy to see that the polynomial $\Psi(z)$ must be linear. Thus $L(z)+\Psi(z)$ is still a linear polynomial in $ z_1,\ldots ,z_n $. Hence, we have $h_1(z)=L(z)+b_1$ and $h_2(z)=L(z)+b_2$. Therefore, the equation \eqref{eq-3.35} becomes
\begin{align}\label{Eq-3.44}
	\begin{cases}
		R_{22}a_1e^{i(L(c)-b_1+b_2)}\equiv 1,\vspace{2mm}\\	R_{22}a_1e^{i(L(c)+b_1-b_2)}\equiv 1.
	\end{cases}
\end{align}
By \eqref{eq-3.21} and the first equation of \eqref{eq-3.35}, we see that
\begin{align*}
	R_{21}\frac{\partial h_1(z)}{\partial z_1}e^{ih_1(z)}-R_{23}e^{ih_2(z+c)}= R_{24},
\end{align*}
or,
\begin{align*}
	R_{21}a_1e^{i(-L(c)+b_1-b_2)}-R_{23}= R_{24}e^{-ih_2(z+c)}.
\end{align*}
Since the left-hand side is constant, the right-hand side must be constant as well. Hence either $R_{24}=0$, which yields $T_2=0$, or $h_2$ is constant, which contradicts our assumptions. Therefore $R_{24}=0$ and consequently $T_2=0$. Hence, we get
\begin{align}\label{Eq-3.45}
	R_{21}a_1e^{i(-L(c)+b_1-b_2)}= R_{23}.
\end{align}
Again, using \eqref{eq-3.34} and the second equation of \eqref{eq-3.35}, it is easy to see that
\begin{align*}
	R_{21}\frac{\partial h_2(z)}{\partial z_1}e^{ih_2(z)}- R_{23}e^{ih_1(z+c)}= R_{24},
\end{align*}
or,
\begin{align*}
	R_{21}a_1e^{i(-L(c)-b_1+b_2)}-R_{23}= R_{24}e^{-ih_1(z+c)},
\end{align*}
which implies that $R_{24}$ must be $0$. Otherwise, $h_1$ must be a constant, which is a contradiction. Therefore $R_{24}=0$, implies that $T_2=0$.  Hence, we get
\begin{align}\label{Eq-3.46}
	R_{21}a_1e^{i(-L(c)-b_1+b_2)}= R_{23}.
\end{align}
In view of \eqref{Eq-3.44}, \eqref{Eq-3.45} and \eqref{Eq-3.46} simple computation shows that
\begin{align*}
	a^2_1\equiv\frac{R_{23}}{R_{21}R_{22}},\;\;e^{2i(b_1-b_2)}\equiv 1 \;\;\mbox{and}\;\; e^{2iL(c)}\equiv\frac{R_{21}}{R_{22}R_{23}}.
\end{align*}
By \eqref{eq-3.28} and \eqref{eq-3.30}, it is easy to see that 
\begin{align*}
	f_1(z)&=D_{11}\cos h_1(z-c) -D_{12}\sin h_1(z-c)+T_1,\\& =D_{11}\cos(L(z)-L(c)+b_1) -D_{12}\sin(L(z)-L(c)+b_1)+T_1
\end{align*}
and
\begin{align*}
	f_2(z)&=D_{11}\cos h_2(z-c) -D_{12}\sin h_2(z-c)+T_1,\\& =D_{11}\cos(L(z)-L(c)+b_2) -D_{12}\sin(L(z)-L(c)+b_2)+T_1.
\end{align*}
	
\noindent{\bf Case 2.} Suppose that
\begin{align*}
	\begin{cases}
		R_{22}\dfrac{\partial h_1(z)}{\partial z_1}e^{i(h_2(z+c)-h_1(z))}\equiv 1,\vspace{2mm}  \\R_{21}\dfrac{\partial h_2(z)}{\partial z_1}e^{i(h_1(z+c)+h_2(z))}\equiv 1,
	\end{cases}
\end{align*}
The above equations yield that $h_2(z+c)-h_1(z)\equiv\xi_1  $ and $ h_1(z+c)+h_2(z)\equiv \xi_2 $, where $ \xi_1,\xi_2 $ are two constants in $ \mathbb{C} $. It is easy to see that $ h_2(z+2c)+h_2(z)\equiv \xi_1+\xi_2 $, which is a contradiction as $ h_2(z) $ is a non-constant polynomial in $ \mathbb{C}^n $.\vspace{1.2mm}
	
\noindent{\bf Case 3.} Let
\begin{align*}
	\begin{cases}
		R_{21}\dfrac{\partial h_1(z)}{\partial z_1}e^{i(h_2(z+c)+h_1(z))}\equiv 1,\vspace{2mm}\\ R_{22}\dfrac{\partial h_2(z)}{\partial z_1}e^{i(h_1(z+c)-h_2(z))}\equiv 1.	
	\end{cases}
\end{align*}
By using the similar argument, we see that $h_2(z+c)+h_1(z)\equiv\xi_1  $ and $ h_1(z+c)-h_2(z)\equiv \xi_2 $, where $ \xi_1,\xi_2 $ are two constants in $ \mathbb{C} $. It is easy to see that $ h_1(z+2c)+h_1(z)\equiv \xi_1+\xi_2 $, which is a contradiction as $ h_1(z) $ is a non-constant polynomial in $ \mathbb{C}^n $.\vspace{1.2mm}
	
\noindent{\bf Case 4.} Suppose that
\begin{align}\label{eq-3.36}
	\begin{cases}
		R_{21}\dfrac{\partial h_1(z)}{\partial z_1}e^{i(h_2(z+c)+h_1(z))}\equiv 1,\vspace{2mm}\\ R_{21}\dfrac{\partial h_2(z)}{\partial z_1}e^{i(h_1(z+c)+h_2(z))}\equiv 1.
	\end{cases}
\end{align}
Therefore, it follows that $ h_2(z+c)+h_1(z)\equiv \xi_1 $ and $ h_1(z+c)+h_2(z)\equiv \xi_2 $, where $ \xi_1 $ and $ \xi_2 $ are two constants in $ \mathbb{C} $. This implies that $ h_1(z+2c)-h_1(z)\equiv \xi_2-\xi_1 $ and $ h_2(z+2c)-h_2(z)\equiv \xi_1-\xi_2 $. Now we claim that $ h_1(z)=L(z)+b^{*}_1 $ and $ h_2(z)=-L(z)+b^{*}_2 $, where $ L(z) $ is a linear function as the form $ L(z)=a_1z_1+\cdots+a_nz_n $ and $ b^{*}_1,b^{*}_2 $ are two constants in $ \mathbb{C} $. In fact, since $ h_1(z),h_2(z) $ are two non-constant polynomials and  $ h_k(z+2c)-h_k(z)$ are two constant, we conclude that $ h_1(z)=L(z)+\Psi(z)+b^{*}_1 $ and $ h_2(z)=-L(z)-\Psi(z)+b^{*}_2 $, where $\Psi(z)$ is a polynomial defined in \eqref{Eq-3.1}. Now from $ \eqref{eq-3.36} $ it is easy to see that the polynomial $\Psi(z)$ must be linear. Thus $L(z)+H(s)$ is still a linear polynomial in $ z_1,\ldots,z_n $. Hence, we have $ h_1(z)=L(z)+b^{*}_1 $ and $ h_2(z)=-L(z)+b^{*}_2 $. Thus, the equation \eqref{eq-3.36} becomes
\begin{align}\label{Eq-3.48}
	\begin{cases}
		R_{21}a_1e^{i(-L(c)+b^{*}_1+b^{*}_2)}\equiv 1,\vspace{2mm}\\	-R_{21}a_1e^{i(L(c)+b^{*}_1+b^{*}_2)}\equiv 1.
	\end{cases}
\end{align}
Considering \eqref{eq-3.32} and the first equation of \eqref{eq-3.36}, we have
\begin{align*}
	R_{22}\frac{\partial h_1(z)}{\partial z_1}e^{-ih_1(z)}- R_{23}e^{ih_2(z+c)}= R_{24},
\end{align*}
or,
\begin{align*}
	R_{22}a_1e^{i(L(c)-b^{*}_1-b^{*}_2)}-R_{23}= R_{24}e^{-ih_2(z+c)},
\end{align*}
which implies that $R_{24}$ must be $0$. Otherwise, $h_2$ must be a constant, which is a contradiction. Therefore $R_{24}=0$ and consequently $T_2=0$.  Hence, we get
\begin{align}\label{Eq-3.49}
	R_{22}a_1e^{i(L(c)-b^{*}_1-b^{*}_2)}= R_{23}.
\end{align}
Again using \eqref{eq-3.34} and the second equation of \eqref{eq-3.36}, we see that
\begin{align*}
	R_{22}\frac{\partial h_2(z)}{\partial z_1}e^{-ih_2(z)}- R_{23}e^{ih_1(z+c)}= R_{24},
\end{align*}
or,
\begin{align*}
	-R_{22}a_1e^{i(-L(c)-b^{*}_1-b^{*}_2)}-R_{23}= R_{24}e^{-ih_1(z+c)},
\end{align*}
which implies that $R_{24}$ must be $0$. Otherwise, $h_1$ must be a constant, which is a contradiction. Consequently $T_2=0$.  Hence, we get
\begin{align}\label{Eq-3.50}
	-R_{22}a_1e^{i(-L(c)-b^{*}_1-b^{*}_2)}= R_{23}.
\end{align}
Thus, it follows from \eqref{Eq-3.48}, \eqref{Eq-3.49} and \eqref{Eq-3.50} that 
\begin{align*}
	a^2_1\equiv \frac{R_{23}}{R_{21}R_{22}},\;\; e^{2i(b^{*}_1+b^{*}_2)}\equiv-\frac{R_{22}}{R_{21}R_{23}}\;\;\mbox{and}\;\; e^{2iL(c)}\equiv -1.
\end{align*}
In view of $ \eqref{eq-3.28} $ and $ \eqref{eq-3.30} $, a simple computation shows that  
\begin{align*}
	f_1(z)&=D_{11}\cos h_1(z-c) -D_{12}\sin h_1(z-c)+T_1,\\& =D_{11}\cos(L(z)-L(c)+b^{*}_1) -D_{12}\sin(L(z)-L(c)+b^{*}_1)+T_1
\end{align*}
and
\begin{align*}
	f_2(z)&=D_{11}\cos h_2(z-c) -D_{12}\sin h_2(z-c)+T_1,\\& =D_{11}\cos(-L(z)+L(c)+b^{*}_2) -D_{12}\sin(-L(z)+L(c)+b^{*}_2)+T_1,\\&  =D_{11}\cos(L(z)-L(c)-b^{*}_2)+ D_{12}\sin(L(z)-L(c)-b^{*}_2) +T_1.
\end{align*}
Assuming $ b_1=b^{*}_1 $ and $ b_2=-b^{*}_2 $, we see that $ e^{2i(b_1-b_2)}\equiv-{R_{22}}/{R_{21}R_{23}} $. Hence, the solutions take the form
\begin{align*}
	f_1(z)=D_{11}\cos(L(z)-L(c)+b_1)- D_{12}\sin(L(z)-L(c)+b_1)+T_1,\\   f_2(z) =D_{11}\cos(L(z)-L(c)+b_2) + D_{12}\sin(L(z)-L(c)+b_2) +T_1.
\end{align*}
This completes the proof.
\end{proof}

\begin{rem}
Theorem \ref{th-2.3} improves several existing results such as \cite[Theorem 1.2]{XU-CAO-MJM-2018}, \cite[Theorem 2.3]{XU-JIANG-RACSAM-2022} and \cite[Theorem 1.3]{Xu-Liu-Li-JMAA-2020} for a system of general quadratic differential-difference equations in $\mathbb{C}^n$.
\end{rem}

By the following example, we confirm that existence of the transcendental entire solutions with finite order for the system of equation in Theorem \ref{th-2.3}.
\begin{exm}
For $c=(\sqrt{3}, -2, 1)\in\mathbb{C}^3$ and $ b_1=\left(1+\ln \left((10i +\sqrt{2})/(10i -\sqrt{2})\right)\right)/(2i)$, and $ b_2=1/(2i) $, the following functions $(f_1,f_2)$, where
\begin{align*}
	f_1(z)&=\frac{4\sqrt{3}}{\sqrt{204\pm 44\sqrt{17}}}\cos\left(\frac{ \sqrt{2}}{\sqrt{3}}z_1 -\frac{\pi}{4}z_2 -\sqrt{2} z_3 -\frac{\pi}{2}
	+\frac{1}{2i}\left(1+\ln\left(\frac{10i +\sqrt{2}}{10i -\sqrt{2}}\right)\right)\right)\\&\quad - \frac{\sqrt{3}(1\pm \sqrt{17})}{\sqrt{136\mp 24\sqrt{17}}} \sin\left(\frac{ \sqrt{2}}{\sqrt{3}}z_1 -\frac{\pi}{4}z_2 -\sqrt{2} z_3 -\frac{\pi}{2}
	+\frac{1}{2i}\left(1+\ln\left(\frac{10i +\sqrt{2}}{10i -\sqrt{2}}\right)\right)\right)+ \frac{1}{2}
\end{align*}
and 
\begin{align*}
	f_2(z)&=\frac{4\sqrt{3}}{\sqrt{204\pm 44\sqrt{17}}}\cos\left(\frac{ \sqrt{2}}{\sqrt{3}}z_1 -\frac{\pi}{4}z_2 -\sqrt{2} z_3 -\frac{\pi}{2}
	+\frac{1}{2i}\right)\\&\quad + \frac{\sqrt{3}(1\pm \sqrt{17})}{\sqrt{136\mp 24\sqrt{17}}} \sin\left(\frac{ \sqrt{2}}{\sqrt{3}}z_1 -\frac{\pi}{4}z_2 -\sqrt{2} z_3 -\frac{\pi}{2}
	+\frac{1}{2i}\right)+ \frac{1}{2}
\end{align*}
is a pair of transcendental entire solutions in $\mathbb{C}^3$ of the system 
\begin{align*}
	\begin{cases}
		2f_1(z+c)^2 + 4 f_1(z+c)\dfrac{\partial f_2(z)}{\partial z_1}+ 3\left(\dfrac{\partial f_2(z)}{\partial z_1}\right)^2 -2f_1(z+c) -2\dfrac{\partial f_2(z)}{\partial z_1}=1,\vspace{1.2mm}\\ 2f_2(z+c)^2 + 4 f_2(z+c)\dfrac{\partial f_1(z)}{\partial z_1}+ 3\left(\dfrac{\partial f_1(z)}{\partial z_1}\right)^2 - 2f_2(z+c) -2\dfrac{\partial f_1(z)}{\partial z_1}=1.
	\end{cases}
\end{align*}
\end{exm}

As a consequence of Theorem \ref{th-2.3}, we derive the following corollary, which broadens the scope of \cite[Theorem 2.3]{XU-JIANG-RACSAM-2022} by allowing for arbitrary constants in the equations of the system, and also generalize from $\mathbb{C}^2$ to $\mathbb{C}^n$ for $n\geq 2$.
\begin{cor}\label{cor-2.3}
Let $  c=(c_1\ldots,,c_n)\setminus\{(0,\ldots, 0)\}\in\mathbb{C}^n $, $d_1\in\mathbb{C}\setminus\{0\}$ and $ a,b,\alpha\in\mathbb{C} $ with $\alpha^2\neq 0, ab $. Then any pair of finite order transcendental entire solutions in $\mathbb{C}^n$ for the system of partial differential-difference equations
\begin{align}\label{eq-2.5}
		\begin{cases}
			af_1(z+c)^2+2\alpha f_1(z+c)\dfrac{\partial f_2(z)}{\partial z_1}+b\left(\dfrac{\partial f_2(z)}{\partial z_1}\right)^2=1, \vspace{1.2mm}\\ af_2(z+c)^2+2\alpha f_2(z+c)\dfrac{\partial f_1(z)}{\partial z_1}+b\left(\dfrac{\partial f_1(z)}{\partial z_1}\right)^2=1,
		\end{cases}
\end{align}
	must be one of the forms
\begin{enumerate}
		\item [(i)]$ f_1(z)=A_{11}\cos[L(z)-L(c)+b_1] -A_{12}\sin[L(z)-L(c)+b_1] $,\vspace{2mm}\\ 
		$ f_2(z)=A_{11}\cos[L(z)-L(c)+b_2] -A_{12}\sin[L(z)-L(c)+b_2] $,
\end{enumerate}
where $ L(z)=\sum_{r=1}^{n} a_r z_r $ and $ b_1,b_2\in\mathbb{C} $ satisfy
	\begin{align*}
		a^2_1\equiv\frac{R_{23}}{R_{21}R_{22}},\;\;e^{2i(b_1-b_2)}\equiv 1 \;\;\mbox{and}\;\; e^{2iL(c)}\equiv\frac{R_{21}}{R_{22}R_{23}},
	\end{align*}
	\begin{align*}
		R_{21}=\frac{-A_{11}-iA_{12}}{iB_{11}-B_{12}},\;\; R_{22}=\frac{A_{11}-iA_{12}}{iB_{11}-B_{12}},\;\; R_{23}=\frac{iB_{11}+B_{12}}{iB_{11}-B_{12}}.
	\end{align*}
Also $ A_{11}, A_{12}, B_{11}\;\mbox{and}\; B_{12} $ are defined in $ (i) $ of Corollary $ \ref{cor-2.1} $.\vspace{1.2mm}
\begin{enumerate}
	\item [(ii)] $ f_1(z)=A_{11}\cos[L(z)-L(c)+b_1]- A_{12}\sin[L(z)-L(c)+b_1],\vspace{2mm}\\
	  f_2(z) =A_{11}\cos[L(z)-L(c)+b_2] + A_{12}\sin[L(z)-L(c)+b_2] $,
\end{enumerate}
where $ A_{11},A_{12} $ and $L(z)$ are stated as in $ (i) $ and $ b_1, b_2 \in\mathbb{C} $ satisfy
\begin{align*}
	a^2_1\equiv \frac{R_{23}}{R_{21}R_{22}},\;\; e^{2i(b_1-b_2)}\equiv-\frac{R_{22}}{R_{21}R_{23}}\;\;\mbox{and}\;\; e^{2iL(c)}\equiv -1,
\end{align*}
where $ R_{21},R_{22}\;\mbox{and}\;R_{23} $ are stated in $ (i) $ of Corollary \ref{cor-2.3}.
\end{cor}

\noindent{\bf Acknowledgment:} The authors would like to thank the referee(s) for their helpful suggestions and comments for the improvement of the exposition of the article.\\

\noindent\textbf{Compliance of Ethical Standards:}\\

\noindent\textbf{Conflict of interest.} The authors declare that there is no conflict  of interest regarding the publication of this article.\vspace{1.5mm}

\noindent\textbf{Data availability statement.}  Data sharing is not applicable to this article as no datasets were generated or analyzed during the current study.

\end{document}